\documentclass[10pt,a4paper]{amsart}
\usepackage[toc,page]{appendix}
\usepackage{a4}
\usepackage[active]{srcltx}
\usepackage{amsmath,bbm,esint}
\usepackage{amssymb}
\usepackage{amsfonts}
\usepackage{mathrsfs} 
\usepackage{graphicx}
\usepackage{tikz}
\usepackage{color}
\usepackage[colorlinks=true,urlcolor=blue,
citecolor=red,linkcolor=blue,linktocpage,pdfpagelabels,bookmarksnumbered,bookmarksopen]{hyperref}

\def\a{\alpha}
\def\b{\beta}
\def\d{\delta}

\def\e{\varepsilon}

\newcommand{\cH}{{\mathcal H}}

\newcommand{\Z}{{\mathbb Z}}
\newcommand{\R}{{\mathbb R}}

\newcommand{\T}{\mathbb{T}}
\newcommand{\N}{\mathbb{N}}

\newcommand{\ov}{\overline}

\newcommand{\sP}{{\mathscr P}}

\newcommand{\sQ}{{\mathscr Q}}

\newcommand{\ds}{\displaystyle}

\newcommand{\diver}{\nabla\cdot}

\newcommand{\dd}{\hspace{0.7pt}{ d}}

\newcommand{\Tan}{{\rm Tan}}

\newcommand{\be}{\begin{equation}}
\newcommand{\ee}{\end{equation}}

\newcommand{\ba}{\begin{array}}
\newcommand{\ea}{\end{array}}

\newtheorem{remark}{\textbf{Remark}}[section]
\newtheorem{theorem}{\textbf{Theorem}}[section]

\newtheorem{lemma}[theorem]{\textbf{Lemma}}
\newtheorem{corollary}[theorem]{\textbf{Corollary}}

\newtheorem{definition}[remark]{\textbf{Definition}}

\newtheorem{innercustomgeneric}{\customgenericname}
\providecommand{\customgenericname}{}
\newcommand{\newcustomtheorem}[2]{%
  \newenvironment{#1}[1]
  {%
   \renewcommand\customgenericname{#2}%
   \renewcommand\theinnercustomgeneric{##1}%
   \innercustomgeneric
  }
  {\endinnercustomgeneric}
}

\newcustomtheorem{customthm}{Framework}

\numberwithin{equation}{section}

\newcounter{hyp}
\title[]{Well-posedness of mean field games master equations involving non-separable local Hamiltonians}  

\author[D.M. Ambrose]{David M. Ambrose} 
\address{Department of Mathematics, Drexel University, Philadelphia, PA, USA}
\email{dma68@drexel.edu}

\author[A.R. M\'esz\'aros]{Alp\'ar R. M\'esz\'aros}  
\address{Department of Mathematical Sciences, University of Durham, Durham DH1 3LE, England}
\email{alpar.r.meszaros@durham.ac.uk} 

\date{\today}

\begin{document}

\begin{abstract}
In this paper we construct short time classical solutions to a class of master equations in the presence of non-degenerate individual noise arising in the theory of mean field games. The considered Hamiltonians are non-separable and {\it local} functions of the measure variable, therefore the equation is restricted to absolutely continuous measures whose densities lie in suitable Sobolev spaces. Our results hold for smooth enough Hamiltonians, without any additional structural conditions as convexity or monotonicity. 
\end{abstract}

\maketitle

\section{Introduction}

The theory of Mean Field Games was initiated around the same time by J.-M. Lasry and P.-L. Lions on the one hand (\cite{LasLio-1,LasLio-2,LasLio-3}) and by P. Caines, M. Huang and R. Malham\'e (\cite{HuaMalCai}) on the other hand. The main motivation of both groups was to characterize limits of Nash equilibria of stochastic (or deterministic) differential games, when the number of agents tends to infinity. 

A central object in this theory is the so-called {\it master equation} introduced by P.-L. Lions in his lectures at Coll\`ege de France (cf. \cite{Lio}). This is a nonlocal Hamilton-Jacobi equation set on the space of Borel probability measures, which encodes all the information about the game. One of the main features of the master equation is that it serves as an important tool to show the convergence/mean field limit of Nash equilibria of games with finite number of agents as the number of agents increases to infinity (cf. \cite{CarDelLasLio, DelLacRam:19, DelLacRam:20}). In particular, solutions to the master equation can be used to obtain fine quantitative estimates on the rate of convergence and  these solutions typically provide $\e$-Nash equilibria for games with large but finite number of agents. In the same time, this equation contains all the information about the finite dimensional mean field game system and it also describes a precise quantitative stability of this system with respect to the initial distribution of the agents.

The concept of master equations has a long history in kinetic theory and mean field limits of particle systems (see for instance \cite{MisMou} and the references therein). The past couple of years have witnessed a great increase of literature on master equations arising in the theory of mean field games. Depending on the techniques used in these works to show the well-posedness of the corresponding master equations, one may group these results into three possible categories. We refer to a non-exhaustive list of works as follows:  probabilistic ideas for problems including individual or common noise were used in \cite{ChaCriDel, CarDel-II, GanMesMouZha, MouZha}; variational techniques (based on optimal transport or optimal control theory in Hilbert spaces, for problems without noise or with individual noise) were exploited in \cite{GanSwi:15, May, GanMes, BenGraYam}; and finally PDE techniques were used in \cite{CarDelLasLio, CarCirPor} to attack problems with common or individual noise. In most of these references, a special hypothesis is assumed on the Hamiltonian $\cH$ appearing in the master equation, namely it is such that the momentum variable it is separated from the measure variable, i.e. it has the typical form of $\cH(t,x,p,m):=H(t,x,p)+f(t,x,m)$. Moreover, in all the cases previously considered in the literature, the dependence of the Hamiltonian on the measure variable is {\it always} assumed to be (nonlocal) regularizing. To the best of our knowledge, \cite{CarCirPor, CarDel-II} are the only works, where the authors show the short time well-posedness of master equations involving non-separable regularizing Hamiltonians under mild assumptions (these are supposed to be smooth enough with some additional growth condition of polynomial type in the momentum variable). The recent work \cite{GanMesMouZha} constructs global in time classical solutions to master equations involving non-separable Hamiltonians, under an additional, so-called displacement monotonicity assumption on the Hamiltonian. Finally, \cite{May} constructs local in time classical solutions to the master equation in the deterministic setting for smooth regularizing Hamiltonians that have the separable structure.

\medskip

In this paper, we show the existence and uniqueness of classical solutions of a class of so-called {\it first order} master equations (i.e. driven by non-degenerate individual noise, and so derivatives with respect to the measure variable appear only at first order) with non-separable Hamiltonians that depend {\it locally} on the measure variable. Because of this local dependence, we clearly need to restrict the domain of the master function to the set of absolutely continuous probability measures whose densities lie in a suitable Sobolev space $H^s(\T^d)$ (for some $s>1$). Ours seems to be the first result on the well-posedness of master equations where the Hamiltonian is a local function of the measure variable. The choice of the physical space $\T^d$ is for convenience, to avoid non-compactness issues. Nevertheless, we expect our results to hold true, without major complications, in the setting of $\R^d$ as well (under suitable moment bounds on the measures).

It is a very natural and interesting question whether can one extend the short-time results to global in time ones, in presence of (the generalized version of) the Lasry-Lions monotonicity condition on the data (see for instance condition (1.6) in \cite{AchPor} imposed on non-separable Hamiltonians). This will be studied in a future work.

\medskip

The equation in the center of our focus for $U:[0,T]\times\T^d\times\sP(\T^d)\cap H^s(\T^d)\to\R$ reads as
\begin{equation}\label{eq:master_main}
\left\{
\begin{array}{l}
\ds-\partial_t U(t,x,m)-\Delta U(t,x,m) +\cH(t,x,\nabla U(t,x,m),m)-\ds\int_{\T^d}\nabla_y\cdot(\nabla_w U(t,x,m,y))\dd m(y)\\[7pt]
\ds+\int_{\T^d}\nabla_w U(t,x,m,y)\cdot D_p\cH(t,y,\nabla U(t,y,m))\dd m(y)=0, \quad (t,x,m)\in(0,T)\times\T^d\times\sP(\T^d)\cap H^s(\T^d),\\[7pt]
U(T,x,m)=G(x,m), \quad\quad (x,m)\in\times\T^d\times\sP(\T^d)\cap H^s(\T^d).
\end{array}
\right.
\end{equation} 
The precise assumptions on the data $\cH:[0,T]\times\T^d\times\R^d\times[0,+\infty)\to\R$ and $G:\T^d\times\sP(\T^d)\to\R$ will be given in the next section. In the previous equation the special notation $\nabla_w U$ stands for the Wasserstein gradient of $U$ with respect to the measure variable (see for instance in \cite{AmbGigSav,GanTud, CarDelLasLio}). All other notations for derivatives are understood with respect to the time and spacial variables. 


\medskip

In our analysis we use PDE techniques in $H^s(\T^d)$ -- when we look at the finite dimensional mean field games system as characteristics of the master equation -- and similarly as in \cite{CarDelLasLio} and \cite{CarCirPor}, it is natural to work at the level of the linearized system to obtain the necessary regularity estimates on the master function. Our approach, at the technical level, fundamentally differs from the ones in \cite{CarDelLasLio} and \cite{CarCirPor} (where the regularization effect of the Hamiltonian in the measure variable was indispensable): at the level of this linearized system, in lack of regularization effect of the Hamiltonian in the measure variable, we perform a careful analysis using the energy method adapted to forward-backward problems. These techniques were used previously by the first author in \cite{Amb18:2}, to show the well-posedness of the underlying finite dimensional mean field games system. As a result of this, when showing the regularity of the master function with respect to the measure variable, it turns out to be very convenient to work in the metric space $(\sP(\T^d)\cap H^s(\T^d),\|\cdot\|_{H^{-1}})$. 
In fact, by the Sobolev embedding theorem, all the regularity estimates on the master function in this metric space will imply the corresponding estimates also in $(\sP(\T^d)\cap H^s(\T^d),\|\cdot\|_{H^{s}})$.

Because of the comparison between the $H^{-1}$ and $W_2$ metrics (cf. \cite{Pey}), we expect our results to be useful in studying the convergence problem of games with finitely many players to the mean field limit, in the presence of non-separable, local Hamiltonians. A breakthrough on the convergence problem in the direction of Hamiltonians depending locally on the measure variable has been recently achieved in \cite{Car:17}. Here the proof of the convergence is carried out via the solution of the master equation in which the dependence of the Hamiltonian on the measure variable becomes more and more singular as the number of agents is increasing. Because of this, a careful analysis had to be combined with special structural and monotonicity assumptions on the Hamiltonian. In seems unclear whether such an argument could work directly in our setting, because of the difference in nature of the assumptions on the Hamiltonian. In a future work we aim to pursue a more `classical route' by first showing that the MFG system (for a class of Hamiltonians that depend locally on the measure variable) is well-posed in Sobolev spaces, for general measure initial data. Having such a result in hand (for which some initial investigation seems already promising), would hopefully translate to the level of the master equation and to the convergence problem. We remark again that some of the crucial estimates on the master function are carried out in Sobolev spaces of negative order, which give us hope to be able to extend the solution of the master equation to rougher probability measures.

Finally, let us remark that (finite dimensional) mean field games systems involving non-separable Hamiltonians that depend locally on the measure variable, while they appear very naturally in models coming from economics (cf. \cite{AchBueLasLioMol}), are poorly understood in the literature in general. Beside the works \cite{Amb18:2, Amb18}, models involving so-called congestion effects have been studied in \cite{AchPor, cirantEtAl, EvaFerGomNurVos, EvaGom, GomVos, GomSed}.

\medskip

The structure of the paper is as follows. In Section \ref{sec:prelim} we collect all our standing assumptions and some preliminary results from the literature. This section recalls the notions of derivatives of functions defined on measures that we use in the rest of the paper. In the same time, here we describe the roadmap of our analysis, with the precise steps that lead to the proof of our main theorem. Then in the upcoming Sections \ref{sec:step1}, \ref{sec:step2} and \ref{sec:step3} we provide all the arguments to fill the necessary details on the steps prescribed in Section \ref{sec:prelim}. We end the main text with Subsection \ref{subsec:regularity} where we collect the necessary regularity estimates on the master function and conclude with the existence and uniqueness of a solution to \eqref{eq:master_main}. Finally, in Appendix \ref{sec:appendix_estimate} we provide a technical stability result on the mean field game system, which is the consequence of the results in \cite{Amb18:2}.

\medskip

{\noindent \sc Acknowledgements}

The authors wish to thank W. Gangbo his constant interest in this work and for his valuable feedback and comments that he gave at various stages during the preparation of the manuscript. This project was conceived at IPAM, UCLA and a part of it was done while both authors were members of the long program ``High Dimensional Hamilton-Jacobi PDEs'' in 2020. D.M.A. has been partially supported by the NSF grant DMS-1907684,  A.R.M. has been partially supported by the Air Force grant FA9550-18-1-0502 and by the King Abdullah University of Science and Technology Research Funding (KRF) under Award No. ORA-2021-CRG10-4674.2.

\medskip

\section{Preliminaries, main assumptions and our main theorems}\label{sec:prelim}

Let $\T^d:=\R^d/(2\pi \Z^d)$ stand for the flat torus embedded in $\R^d$. We denote by $\sP(\T^d)$ the space of Borel probability measures on the flat torus $\T^d$, for $s>1$ we set $\sQ:=\sP(\T^d)\cap H^s(\T^d)$. For $R>0$, we denote by $\sQ_R$ the elements of $m\in\sQ$ such that $\left\|m-\ov m\right\|_{H^s}\le R,$ where $\ov m:=\frac{1}{(2\pi)^d}$. These stand for $R$-balls in $H^s(\T^d)$ centered at the uniform density on $\T^d$. Let $T>0$.

Let $\cH:[0,T]\times\T^d\times\R^d\times[0,+\infty)\to\R$ be a given Hamiltonian function and let $G:\T^d\times\sQ\to\R$ be a given final cost. We assume the following hypotheses on the data.

\medskip

{\bf Standing assumptions}

\medskip

\begin{equation}\label{hyp:s}
s>\max\left\{\left\lceil(d+5)/2\right\rceil +1; 4\lceil d/2\rceil{+1}\right\}\ \ {\rm{is\ a\ given\ real\ number}}.\tag{H\arabic{hyp}}
\end{equation}
\stepcounter{hyp}

Since our results are built upon the ones in \cite{Amb18:2}, first we recall all the assumptions present in \cite{Amb18:2}. First, let $\beta\in(\N\cup\{0\})^{2d+1}$ be a multi-index, associated to the arguments $(t,x,p,q)$ of the Hamiltonian $\cH$. The first $d$ coordinates of $\b$ correspond to $(x_1,\dots,x_d)$, the following $d$ coordinates correspond to $(p_1,\dots, p_2)$ (which is the placeholder for $\nabla u$), while the last coordinate corresponds to $q$ (which is the placeholder for the $m$ variable).  Suppose that there exists $\tilde F:[0,+\infty)\to[0,+\infty)$ non-decreasing such that
\begin{align}\label{hyp:H1_original}
\left| \partial^\beta\cH(\cdot,\cdot,\nabla u,m) \right|_\infty\le \tilde F\left(|\nabla u|_{\infty}+|m|_\infty\right),\ \ \forall\ \beta\in\N^{2d+1},\ |\beta|\le s+2.
\tag{H\arabic{hyp}}
\end{align}
\stepcounter{hyp}
Suppose that for any $B\subset \R^{d+1}$ bounded set and $\b\in\N^{2d+1}$ with $0\le |\beta|\le 2$, $\exists c>0$ such that if $(p^1,q^1),(p^2,q^2)$, then
\begin{align}\label{hyp:H3_original}
|\partial^\beta\cH(t,x,p^1,q^1)-\partial^\beta\cH(t,x,p^2,q^2)|\le c\left(\sum_{i=1}^d|p^1_i-p^2_i|+|q^1-q^2|\right),\ \ \forall (t,x)\in [0,T]\times\T^d\tag{H\arabic{hyp}}.
\end{align}


In Section \ref{sec:step3} we will need a precise estimate on $D^2\cH$ and $D^3\cH$. To describe this, let us define the following quantities.
\begin{equation}\nonumber
F_{1}=\mathcal{H}(t,x,\nabla\tilde{u},\tilde{m})-\mathcal{H}(t,x,\nabla u,m)\\
-D_{p}\mathcal{H}(t,x,\nabla u,m)\cdot\nabla(\tilde{u}-u)
-\partial_q\mathcal{H}(t,x,\nabla u,m)(\tilde{m}-m),
\end{equation}
\begin{equation}\nonumber
F_{2}=\mathrm{div}(mD_{p}\mathcal{H})-\mathrm{div}(\tilde{m}\widetilde{D_{p}\mathcal{H}})
-\mathrm{div}((\tilde m-m)D_{p}\mathcal{H})
-\mathrm{div}(m(D_{pp}^{2}\mathcal{H})\nabla(\tilde u-u))
-\mathrm{div}(m(D_p\partial_q\mathcal{H})(m-\tilde{m})),
\end{equation}
where we used the shorthand notations $\widetilde{D_{p}\mathcal{H}}=D_{p}\mathcal{H}(t,x,\nabla\tilde{u},\tilde{m})$ and $D_p\cH=D_{p}\mathcal{H}(t,x,\nabla u,m)$ in the last line.

Let us notice that under the assumption that $s$ satisfies \eqref{hyp:s}, we have that there exists $r>0$ such that 
\begin{equation}\label{con:r}
r>\left\lceil d/2\right\rceil\ \ {\rm{and}}\ \ s\ge 4r+1> \max\left\{\left\lceil(d+5)/2\right\rceil +1; 4\lceil d/2\rceil+1\right\}.
\end{equation}

Let $r>0$ satisfy \eqref{con:r}. We assume that for any $R>0$ then there exists $c>0$ such that
\begin{equation}\label{f1Bound}
\|F_{1}\|_{H^{r}}^{2}\leq c(\|u-\tilde{u}\|_{H^{r+1}}^{4}+\|m-\tilde{m}\|_{H^{r}}^{4}),
\tag{H\arabic{hyp}}
\end{equation}
\stepcounter{hyp}
\begin{equation}\label{f2Bound}
\|F_{2}\|_{H^{r-1}}^{2}\leq c(\|u-\tilde{u}\|_{H^{r}}^{4}+\|m-\tilde{m}\|_{H^{r-1}}^{4}),
\tag{H\arabic{hyp}}
\end{equation}
for all $m,\tilde m, u,\tilde u$ such that $\|m\|_{H^r},\|\tilde m\|_{H^r},\|u\|_{H^{r+1}},\|\tilde u\|_{H^{r+1}}\le R.$
\stepcounter{hyp}

Here and afterwards we simply use the notation $D_p\cH$ and $\partial_q\cH$ to refer to the derivatives of $\cH$ with respect to the third (the placeholder for $\nabla u$) and fourth (the placeholder for $m$) variables, respectively. Let us underline that  $\partial_q\cH$ is used instead of $\partial_m\cH$ to avoid possible confusions with the Wasserstein derivative or $L^2$-G\^ateaux derivative with respect to the measure variable and to emphasize that this is a standard `local' derivative of the Hamiltonian function.

\begin{remark}
The Hamiltonians of the form
$$\cH(t,x,p,q)=a(t,x)P(p)Q(q),$$
where $a:[0,+\infty)\times\T^d$ is smooth with bounded derivatives and $P:\R^d\to \R, Q:\R\to\R$ are polynomial functions, satisfy our standing assumptions \eqref{hyp:H1_original} through \eqref{f2Bound}. Involving transcendental functions in the form
of $\cH$ is also allowable, as $\cH(t,x,p,q)=\sin(|p|^2)\ln(1+q^{2})$ or 
$\cH(t,x,p,q)=\exp\{\cos(|p|^2 q^3)\}$ also satisfy the assumptions.

Functions of finite regularity, such as fractional powers, are also admissible, but must be sufficiently regular.  For instance, if we
have $d=1$ then we find $s=6$ is acceptable in \eqref{hyp:s}.  Then in \eqref{hyp:H1_original}, we must have $\mathcal{H}$ 
differentiable with bounded derivatives at least 8 times; therefore, a Hamiltonian of the form $\cH(t,x,p,q)=|p|^{17/2}m^2$
would be admissible. 

The prior work \cite{Amb18:2} includes a result on Hamiltonians applicable to congestion problems; in such Hamiltonians,
a power of the measure, $m,$ appears in the denominator.  This naturally carries with it an assumption that the initial measure be 
bounded away from zero.  We expect that such Hamiltonians (e.g. $\cH(t,x,p,q)=|p|^{2}/q^{1/2}$) could be treated from the point
of view of the present results on the master equation, at the cost of this further restriction on the class of measures considered.  
We do not do so in the present work, but could pursue this direction in the future.

Finally, we mention that there may be Hamiltonians for which the arguments of \cite{Amb18:2} and the present work 
apply, even though the Hamiltonian may not satisfy exactly \eqref{hyp:H1_original} through \eqref{f2Bound}.  The 
assumptions are in place to work for a very general class of Hamiltonians, and it is likely that for some particular Hamiltonians, 
fewer derivatives may be required.
\end{remark}

\medskip


We consider the final cost function $G$ to be a nonlocal smoothing operator applied to a function of $m.$
This is what was done also in \cite{cirantEtAl} and \cite{CarCirPor}. The reason for the regularization is that the solutions $(u,m)$ to \eqref{eq:mfg} have the following regularity: $u(t,\cdot)\in H^{s}$ and $m(t,\cdot)\in H^{s-1}.$  If $u$ has data $G(m(T,\cdot))$ and $G$ is not regularizing, this is a problem
since $m(T,\cdot)$ is not in the space that $u$ should be in.  Thus we take $G$ to be smoothing.
This issue was 
discussed also in \cite{Amb18:2}, and \cite[Section 5.1]{Amb18:2} presents the corresponding well-posedness result on the MFG system in the case of regularizing final cost functions
(we note that the non-regularizing case is addressed there as well). We assume the following conditions on $G$. 

There exists $\Upsilon>0$ such that for all $m_1,m_2\in H^{s}(\T^d),$ 
\begin{equation}\label{hyp:G-Lip}
\|G(m_{1},\cdot)-G(m_{2},\cdot)\|_{H^1}^{2}\leq\Upsilon\|m_{1}-m_{2}\|_{L^2}^{2}.\tag{H\arabic{hyp}}
\end{equation}
\stepcounter{hyp}

We assume that there exists $\kappa>0$ such that for any $s$ satisfying \eqref{hyp:s}, we have
\begin{equation}\label{payoffAssumption}
\left\|\frac{\delta G}{\delta m}(\cdot,m)\mu\right\|_{H^{s}}^{2}\leq \kappa \|\mu\|_{H^{s-1}}^{2},\ \forall m\in \sQ_R.\tag{H\arabic{hyp}}
\end{equation}
\stepcounter{hyp}
It will be convenient for us to take $\kappa\geq1,$ so we make this assumption. Here, $\frac{\delta G}{\delta m}$ stands for the G\^ateaux derivative of $G$, that we define below. In the same time, by $\frac{\delta G}{\delta m}(\cdot,m)\mu$ we denote the action of $\frac{\delta G}{\delta m}(\cdot,m)$ on $\mu$.

Finally, we assume that $G$ is G\^ateaux differentiable on $H^{s}(\T^d)$ (for $s$ satisfying \eqref{hyp:s}) and for any $R>0$, there exists $c>0$ such that
\begin{equation}\label{hyp:G-C11}
\left\|G(\cdot,\tilde m)-G(\cdot, m)-\frac{\delta G}{\delta m}(\cdot,m)(\tilde m-m)\right\|_{H^s}\le c\|\tilde m-m\|^2_{H^s},\ \ \forall \tilde m,m\in \sQ_R.\tag{H\arabic{hyp}}
\end{equation}

\begin{remark}
The function 
$$
G(x,m):=[W*(g(m))](x),
$$
where $W:\T^d\to\R$ is smooth and $g:\R\to\R$ is bounded and of class $C^{1,1}$ with bounded derivative, satisfies the assumptions \eqref{hyp:G-Lip}-\eqref{hyp:G-C11}. We note that the assumption \eqref{hyp:G-Lip} could be modified to consider $m_{i}$ in a bounded set in $H^{s}(\T^d),$ and then the function $g$ would only need to be locally $C^{1,1}$.
\end{remark}


\subsection{Notions of derivatives of functions defined on $\sP(\T^d)$}

Let $V:\sP(\T^d)\to\R$ and $m_0\in\sP(\T^d)$. We say that $V$ is $L^2$-differentiable in G\^ateaux sense at $m_0$ if there exists $\frac{\delta V}{\delta m}(m_0):\T^d\to\R$ continuous such that 
$$
\lim_{\e\to 0}\frac{V(m_0+\e\chi)-V(m_0)}{\e}=:\int_{\T^d}\frac{\delta V}{\delta m}(m_0)(y)\dd\chi(y)
$$
for all $\chi$ signed Borel measure such that $\chi(\T^d)=0$, independently on $\chi$. Notice that whenever $\frac{\delta V}{\delta m}(m_0)$ exists, it is unique up to additive constants. In what follows it will be convenient to fix such a constant that $\ds\int_{\T^d}\frac{\delta V}{\delta m}(m_0)(y)\dd m_0(y)=0$. For notational convenience, we use the notation $\ds\frac{\delta V}{\delta m}(m_0)\mu:=\int_{\T^d}\frac{\delta V}{\delta m}(m_0)(y)\dd\mu(y)$.

Given $m_0\in\sP(\T^d)$, we denote as $L^2(m_0)$ the subset of Borel fields $v: \T^d \rightarrow \R^d$ which are $m_0$--square integrable. The Wasserstein tangent space at $m_0$, denoted as $\Tan_{m_0}\sP(\T^d)$, is the closure of $\nabla C^\infty(\T^d)$ in $L^2(m_0).$ We say that $V:\sP(\T^d)\to\R$ is differentiable at $m_0$ if it is sub- and super differentiable in the Wasserstein sense at $m_0$  (see for instance \cite{AmbGigSav, GanTud, GanTud14}). In this case we denote by $\nabla_w V(m_0):\T^d\to\R^d$ its so-called {\it Wasserstein gradient}, which is the unique element in the intersection of $\Tan_{m_0}\sP(\T^d)$ and the sub- and super differentials. Notice that if $\frac{\delta V}{\delta m}(m_0)$ exists and it is in  $C^1(\T^d)$, we simply have that $\nabla_w V(m_0)(y)=\nabla_y\frac{\delta V}{\delta m}(m_0,y).$

\begin{definition}\label{def:gateaux_H-s}
Let $s>0$ and let $R>0$. We say that a function $V:\sQ_{R}\to\R$ is G\^ateaux differentiable on $H^s(\T^d)$ if for all $m\in \sQ_{R}$, $\frac{\d V(m)}{\d m}:\T^d\to\R$ exists and 
$$
\left| V(\tilde m) -V(m) - \int_{\T^d}\frac{\delta V}{\delta m}(m)(y) \dd(\tilde m-m)(y) \right|  = o\left(\|m_0-\tilde m_0\|_{H^{s}}\right),
$$
for all $\tilde m, m\in\sQ_{R}$. 
\end{definition}

\medskip

\begin{definition}\label{def:solution}
Let $s$ satisfy \eqref{hyp:s} and let $R>0$. We say that $U:[0,T]\times\T^d\times\sQ_R\to\R$ is a solution to the master equation \eqref{eq:master_main} if for all $m\in\sQ_{ R}$, $U(\cdot,\cdot,m)\in C^1([0,T]\times\T^d)\cap C([0,T]; C^{2,\alpha}(\T^d))$, $U(t,x,\cdot)$ is G\^ateaux differentiable on $H^s(\T^d)$ (in the sense of Definition \ref{def:gateaux_H-s}), with $\frac{\delta U}{\delta m} (t,x,m)(\cdot)\in H^{s}(\T^d)$, uniformly with respect to $(t,x)\in [0,T]\times\T^d$ 
and ${ \sQ_R}\ni m\mapsto\frac{\delta U}{\delta m}(t,\cdot,m)(y)\in H^{-s}(\T^d)$ is continuous, for all $(t,y)\in (0,T)\times\T^d$. Moreover, \eqref{eq:master_main} is satisfied pointwise for all $(t,x,m)\in (0,T)\times\T^d\times\sQ_R$.
\end{definition}

Our main theorem of this paper can be formulated as follows.

\begin{theorem}\label{thm:main} Let $s, \cH$ and $G$ satisfy \eqref{hyp:s}-\eqref{hyp:G-C11} and let $R>0$. Then, there exists $T_{***}>0$ (depending on the data $\cH$ and $G$ and $R$) such that the master equation \eqref{eq:master_main} has a unique solution on $[0,T_{***}]\times\T^d\times\sQ_R$ in the sense of Definition \ref{def:solution}.

\end{theorem}

\subsection{The strategy of the proof of our main theorem}

In order to prove Theorem \ref{thm:main}, similarly as in \cite[Chapter 3]{CarDelLasLio}, we rely on the well-posedness of a finite dimensional MFG system, its linearization and the smoothness of the value function with respect to the initial measure. Let us mention, however, that even though the roadmap leading to the main results is following \cite{CarDelLasLio},  our analysis is fundamentally different from the one present in this reference. 

Let us recall the main result on the well-posedness of the MFG system, which is the basis of our analysis. This can be found in \cite[Theorem 7]{Amb18:2}. 
\begin{theorem}\label{thm:david_old}
Suppose that the assumptions \eqref{hyp:s}-\eqref{hyp:G-C11} take place and let $R>0$. There exists $\tilde T>0$ (depending on $R$ and the data) such that for all $0\le t_0<\tilde T$ the mean field games system 
\begin{equation}\label{eq:mfg}
\left\{
\begin{array}{ll}
-\partial_t u - \Delta u +\cH(t,x,\nabla u,m)=0, & (t,x)\in(t_0,\tilde T)\times\T^d,\\[7pt]
\partial_t m - \Delta m - \diver(m D_p\cH(t,x,\nabla u,m))=0, & (t,x)\in(t_0,\tilde T)\times\T^d,\\[7pt]
m(t_0,x)=m_0(x),\ \ u(\tilde T,x)=G(x,m_{\tilde T}(x)), & x\in\T^d
\end{array}
\right.
\end{equation}
has a unique classical solution $(u,m)$ for any $m_0\in \sQ_R$. Moreover, this solution has the regularity 
\begin{align*}
u\in L^\infty([t_0,\tilde T]; H^{s}(\T^d))\cap L^2([t_0,\tilde T]; H^{s+1}(\T^d))\cap C([t_0,\tilde T]; H^{s'}(\T^d)), \ \ \forall s'\in [0,s)
\end{align*}
and 
\begin{align*}
m \in L^\infty([t_0,\tilde T]; H^{s-1}(\T^d))\cap L^2([t_0,\tilde T]; H^{s}(\T^d))\cap C([t_0,\tilde T]; H^{s'-1}(\T^d)), \ \ \forall s' \in [0,s)
\end{align*}
and $(u,m)$ is uniformly bounded in the corresponding spaces by a constant depending only on $R>0$, $\cH$ and $G$.
\end{theorem}

\begin{corollary}\label{cor:diff_u_m}
By the Sobolev embedding theorem, one has 
\begin{align*}
u\in C([t_0,\tilde T]; C^{2,\a}(\T^d)) \ \ {\rm{and}}\ \ m \in C([t_0,\tilde T]; C^{2,\alpha}(\T^d)).
\end{align*}
for some $\a\in[0,1)$. Using the fact that $(u,m)$ is a solution to the system \eqref{eq:mfg}, this further yields that one also has $u,m\in C^1([t_0,\tilde T]\times\T^d)$.
\end{corollary}

The candidate for the solution to \eqref{eq:master_main} can be defined as
\begin{equation}\label{def:U}
U(t_0,x,m_0):=u(t_0,x),
\end{equation} 
where $(u,m)$ is the unique solution to \eqref{eq:mfg} with initial measure $m_0$ and initial time $t_0$. We aim to show $U$ solves \eqref{eq:master_main} in the sense of Definition \eqref{def:solution}. 

For this, we follow the following major steps. Let $(u,m)$ be a solution to \eqref{eq:mfg}.

\medskip

{\bf Step 1.} We linearize \eqref{eq:mfg} around this solution $(u,m)$ allowing perturbations of the form $m_0+\e\mu_0$ on the initial data  (here $\mu_0$ is typically a finite signed measure with 0 mass). This system reads as 

\begin{equation}\label{eq:linearized}
\left\{
\begin{array}{ll}
\ds-\partial_t v - \Delta v +D_p\cH(t,x,\nabla u,m)\cdot\nabla v+\partial_q\cH(t,x,\nabla u,m)\mu=0, & (t,x)\in(t_0,T)\times\T^d,\\[7pt]
\partial_t \mu - \Delta \mu - \diver\left[\mu D_p\cH(t,x,\nabla u,m)+mD_{pp}^2\cH(t,x,\nabla u,m)\nabla v\right] & \\ [7pt]
\ds\hspace{1.5cm}+\diver \left[mD_p\partial_q\cH(t,x,\nabla u,m)\mu\right]=0, & (t,x)\in(t_0,T)\times\T^d,\\[7pt]
\ds\mu(t_0,\cdot)=\mu_0,\ \ v(T,x)=\frac{\delta G}{\delta m}(x,m_T)\mu_T, & x\in\T^d.
\end{array}
\right.
\end{equation}
We show that this linearized system has a solution $(v,\mu)$ in suitable Sobolev spaces. In particular we have the following theorem.

\begin{theorem}\label{thm:linearized_intro}
Let $R>0$, $m_0\in \sQ_R$ and let $(u,m)$ be the solution to \eqref{eq:mfg} (given in Theorem \ref{thm:david_old}), let $T>0$. Then, there exists $T_{*}>0$ such that
\begin{itemize}
\item[(1)] if $\mu_0\in H^{s-1}(\T^d)$ and $0<T<T_{*},$ then there exist a unique pair $(v,\mu)\in L^{\infty}([0,T];H^{s}(\T^d))\times L^{\infty}([0,T]; H^{s-1}(\T^d))$ that solves \eqref{eq:linearized};
\item[(2)] if  $\mu_0\in H^{-s-1}(\T^d)$ and $0<T<T_{*},$ then there exists a unique solution $(v,\mu)$ to \eqref{eq:linearized} such that $v\in L^{\infty}([0,T];H^{-s}(\T^d))$ and $\mu\in L^{\infty}([0,T];H^{-s-1}(\T^d))$. This solution is understood in the sense of distributions. 
\end{itemize}
\end{theorem}

The proof of part (1) of this theorem is provided in Section \ref{sec:step1}. Part (2) is a consequence of results in Section \ref{sec:step2} (we detail this in Corollary \ref{cor:weaker_v_mu}).

\medskip

{\bf Step 2.} We show that if $(v,\mu)$ is a solution to \eqref{eq:linearized} with initial condition $\mu_0$, then 
the operator $\mu_0\mapsto v$ is continuous and linear and there exists $K(t_0,x,m_0,\cdot):\T^d\to\R$ such that
$$v(t_0,x)=\int_{\T^d}K(t_0,x,m_0,y)\dd\mu_0(y).$$
This way we have a candidate for the $L^2$-G\^ateaux derivative of the master function, i.e.  
\begin{equation}\label{def:delta_U-K}
K(t_0,x,m_0,\cdot)=\frac{\delta U}{\delta m}(t_0,x,m_0,\cdot).
\end{equation}
This result is obtained as a consequence of a Riesz-type representation theorem. Let us comment on the regularity of $K(t_0,\cdot,m_0,\cdot)$. As a consequence of Theorem \ref{thm:linearized_intro}(2), 
$K$ will have $H^s$ regularity in the last variable. However, let us underline that for $(t_0,m_0,y)\in (0,T_*)\times\sQ_R\times\T^d$ fixed, the application $x\mapsto K(t_0,x,m_0,y)$ will a priori have only $H^{-s}$ regularity. Therefore, a special care is needed in the derivation of the master equation. This is in contrast with the corresponding results from \cite{CarDelLasLio}, where as a consequence of the regularizing effect of the Hamiltonian in the measure variable,  $\frac{\delta U}{\delta m}$ is shown to be smooth also in the second variable.

We provide the details on these arguments in Section \ref{sec:step2} and Subsection \ref{subsec:regularity}.


\medskip

{\bf Step 3.} Lastly, by a  Taylor expansion argument, we show that the kernel $K$ obtained in the previous step is indeed corresponding to the $L^2$-G\^ateaux derivative $\frac{\d U}{\d m}$, i.e. \eqref{def:delta_U-K} is shown rigorously in a suitable sense. To achieve this, we argue as follows. Let $u(t_0,x)$ be the value function in \eqref{eq:mfg} with initial measure $m_0$ and $\tilde u(t_0,x)$ the value function in \eqref{eq:mfg} with initial measure $\tilde m_0$. Let moreover $v(t_0,x)$ be the solution of the first equation in \eqref{eq:linearized} with $\mu_0:=\tilde m_0-m_0$. Then we show that 
$$\| \tilde u(t_0,\cdot) -u(t_0,\cdot) - v(t_0,\cdot) \|_{H^r}= o\left(\|m_0-\tilde m_0\|_{H^{-1}}\right),$$
for $r>0$ satisfying \eqref{con:r}, uniformly with respect to $t_0$.

Actually this further implies that 
$$\left\| U(t_0,\cdot,\tilde m_0) -U(t_0,\cdot,m_0) - \int_{\T^d}K(t_0,\cdot,m_0,y)\dd(\tilde m_0-m_0) \right\|_{H^r}= o\left(\|m_0-\tilde m_0\|_{H^{-1}}\right),$$
uniformly with respect to $t_0$. So in particular, since $r>0$ satisfies \eqref{con:r}, the Sobolev embedding theorem yields 
$$\sup_{t_0\in[0,T],x\in\T^d}\left| U(t_0,x,\tilde m_0) -U(t_0,x,m_0) - \int_{\T^d}K(t_0,x,m_0,y)\dd(\tilde m_0-m_0) \right|= o\left(\|m_0-\tilde m_0\|_{H^{-1}}\right),$$
i.e. the necessary differentiability property of $U$ with respect to the measure variable.

To perform the analysis in this step, we will rely on an important observation. Using the notation 
\begin{equation}\label{def:z}
z(t,x):=\tilde u(t,x)-u(t,x)-v(t,x),
\end{equation}
we will have that $z$ solves the equation
\begin{equation}\label{eq:z}
\left\{
\begin{array}{rl}
-\partial_t z-\Delta z +D_p\cH(t,x,\nabla u,m)\cdot\nabla z =&-\cH(t,x,\nabla\tilde u,\tilde m)+ \cH(t,x,\nabla u, m)
+\partial_q\cH(t,x,\nabla u,m)\mu\\[4pt]
& + D_p \cH(t,x,\nabla u,m)\cdot(\nabla \tilde u-\nabla u),\\[4pt]
z(T,\cdot)=&G(x,\tilde m_T)-G(x, m_T)-\frac{\delta G}{\delta m}(x,m_T)\mu_T,
\end{array}
\right.
\end{equation}
and thus, we essentially show that there exists $C>0$ such that
$$\|z(t,\cdot)\|_{H^r}\le C \|m_0-\tilde m_0\|_{H^{-1}}^{\frac54}.$$
In fact the Sobolev embedding theorem further implies that 
$$\|z(t,\cdot)\|_{H^r}\le C \|m_0-\tilde m_0\|_{H^{s}}^{\frac54},$$
so the differentiability property holds true in $H^s(\T^d)$. We provide the details of this step in Section \ref{sec:step3}, where the previous crucial estimate is provided by Theorem \ref{thm:z_regularity}.

\subsection{Some preliminary estimates and Sobolev norms}

We define the operator $\Lambda$ to be a Fourier multiplier operator with symbol
\begin{equation}\nonumber
\mathcal{F}\Lambda(k)=\left(1+|k|^{2}\right)^{1/2},
\end{equation}
so that $\Lambda^{-1}$ is the operator with symbol
\begin{equation}\nonumber
\mathcal{F}\Lambda^{-1}(k)=\frac{1}{\left(1+|k|^{2}\right)^{1/2}}.
\end{equation}
For $f \in L^1(\mathbb T^d)$ we define 
\[
\hat f(k)= \mathcal F [f](k):= {1\over (2 \pi)^d } \int_{\mathbb T^d} e^{- i k \cdot x} f(x) dx, \qquad \forall k \in \mathbb Z^d.
\]
For $l \in \mathbb R$ and $f \in H^l(\mathbb T^d)$ we define 
\[
\Lambda^l f(x)= \sum_{k \in \mathbb Z^d} |\hat f(k)| { (1+|k|^2)^{l/2}} e^{i k\cdot x}.
\]
We have the norm 
\begin{equation}\label{-1norm}
 \|f\|^2_{H^l}= \sum_{k \in \mathbb Z^d} |\hat f(k)|^2 (1+|k|^2)^{l}=\langle \Lambda^{l}f,\Lambda^{l}f\rangle_{L^{2}}=\|\Lambda^{l}f\|^2_{L^{2}}.
 \end{equation}
 We shall use the following convention: for $v=(v_1, \cdots, v_d) \in (H^l(\T^d))^d$ we set  
 \[
  \|v\|^2_{H^l}:= \sum_{j=1}^d  \|v_i\|^2_{H^l}.
 \]

With this in mind, the identity $\widehat{\nabla f}(k)= ik \hat f(k)$ allows to obtain that 
\[
\|\Lambda^{-s-1} \nabla f\|_{L^2}^{2}+ \|\Lambda^{-s-1} f\|_{L^2}^{2}= \sum_{k \in \mathbb Z^d}{|k|^2|\hat f(k)|^2 \over (1+|k|^2)^{s+1}} + \sum_{k \in \mathbb Z^d}{|\hat f(k)|^2 \over (1+|k|^2)^{s+1}}= \sum_{k \in \mathbb Z^d}{|\hat f(k)|^2 \over (1+|k|^2)^{s}}.
\]
We read off   
\begin{equation}\label{explicitEquivalence}
\|\Lambda^{-s-1} \nabla f\|_{L^2}^{2}+ \|\Lambda^{-s-1} f\|_{L^2}^{2}= \|\Lambda^{-s} f\|_{L^2}^{2}.
\end{equation}

If $f \in H^{-s}(\T^d)$ and $g \in H^s(\T^d)$ we have  
\begin{equation}\label{dualH^s}
\langle f, g\rangle= \langle \Lambda^{-s}f, \Lambda^s g\rangle_{L^2}  
\end{equation}
and so we may conclude
\begin{equation}\label{dualH^s1}
\|f\|_{H^{-s}}= \sup_{\|\phi\|_{H^{s}}=1} \langle f, \phi\rangle.
\end{equation}


Let us recall some important inequalities that will play important roles in our analysis. 


We will need a product estimate for negative-index Sobolev spaces; in particular we want
\begin{equation}\label{negativeProduct}
\|fg\|_{H^{-1}}\leq \|f\|_{H^{-1}}\|g\|_{H^{k}}
\end{equation}
for some $k>0$ large enough. 
We can establish this through duality and the corresponding estimate for positive-index Sobolev spaces.
In particular we have
\begin{equation}\label{negativeBound1}
\|fg\|_{H^{-1}}=\sup_{\|\phi\|_{H^{1}}=1}\langle fg,\phi\rangle_{L^{2}}
=\sup_{\|\phi\|_{H^{1}}=1}\langle f,g\phi\rangle_{L^{2}}\leq\sup_{\|\phi\|_{H^{1}}=1}\|f\|_{H^{-1}}\|g\phi\|_{H^{1}}.
\end{equation}
We then use the corresponding estimate for positive-index spaces (see, for example, \cite{Beale}),
\begin{equation}\label{estim:Hr-Hs}
\|h\psi\|_{H^{r}}\leq c \|h\|_{H^{r}}\|\psi\|_{H^{s}},
\end{equation}
as long as $0\leq r\leq s$ and $s> d/2.$
So then \eqref{negativeBound1} can be estimated further as
\begin{equation}\nonumber
\|fg\|_{H^{-1}}\leq\sup_{\|\phi\|_{H^{1}}=1}c\|f\|_{H^{-1}}\|\phi\|_{H^{1}}\|g\|_{H^{k}}=c\|f\|_{H^{-1}}\|g\|_{H^{k}},
\end{equation}
where our conditions on $k$ are $k\geq1$ and $k>d/2.$  We can take $k=\lceil\frac{d+1}{2}\rceil.$

In the same spirit, for $s$ satisfying \eqref{hyp:s}, we have the inequality 
\begin{equation}\label{estim:H-s}
\|fg\|_{H^{-s}}\le c\|f\|_{H^{-s}}\|g\|_{H^s}.
\end{equation}

We also need the most standard Sobolev interpolation inequality.  Let $0<\alpha<\beta$ be given, and let $f\in H^{\beta}.$
Then the following holds:  
\begin{equation}\label{sobolevInterpolation}
\|f\|_{H^{\alpha}}\leq c\|f\|_{L^{2}}^{1-\alpha/\beta}\|f\|_{H^{\beta}}^{\alpha/\beta}.
\end{equation}
The proof of this can be found many places, including in \cite{ambroseThesis}.

\section{Well-posedness of the linearized system}\label{sec:step1}

\begin{theorem}\label{step1Theorem}
Let $s$ satisfy \eqref{hyp:s}.  Let $T>0$ be given, and let $m_{0}\in H^{s-1}$ be given.  Let 
the payoff function $G$ satisfy \eqref{payoffAssumption}.
Let $u\in L^{\infty}([0,T];H^{s})$ and $m\in L^{\infty}([0,T];H^{s-1})$ solve \eqref{eq:mfg}.  
Let $\mu_{0}\in H^{s-1}$ be given.  There exists $T_{*}>0$ such that if $0<T<T_{*},$ 
then there exist $v\in L^{\infty}([0,T];H^{s})$ and
$\mu\in L^{\infty}([0,T];H^{s-1})$ which satisfy \eqref{eq:linearized}.
\end{theorem}

\begin{proof}
We set up an iterative scheme for $v$ and $\mu,$
\begin{equation}\label{iteratedV}
\partial_{t}v^{n+1}=-\Delta v^{n+1}+\mathcal{J}_{\varepsilon}(D_{p}\mathcal{H}\cdot\nabla v^{n})
+\mathcal{J}_{\varepsilon}(\partial_q\mathcal{H}\mu^{n}),
\end{equation}
\begin{equation}\label{iteratedMu}
\partial_{t}\mu^{n+1}=\Delta\mu^{n+1}+\mathcal{J}_{\varepsilon}\mathrm{div}(\mu^{n}D_{p}\mathcal{H})
+\mathcal{J}_{\varepsilon}\mathrm{div}(mD_{pp}^{2}\mathcal{H}\nabla v^{n})
+\mathcal{J}_{\varepsilon}\mathrm{div}(mD_p\partial_q\mathcal{H}\mu),
\end{equation}
with mollified data
\begin{equation}\nonumber
v^{n+1}(T,\cdot)=\mathcal{J}_{\varepsilon}\left(\frac{\delta G}{\delta m}\mu^{n}(T,\cdot)\right),\qquad
\mu^{n+1}(0,\cdot)=\mathcal{J}_{\varepsilon}\mu_{0}.
\end{equation}
For parameter $\varepsilon>0,$ the operator $\mathcal{J}_{\varepsilon}$ is a Friedrichs mollifier.
We have suppressed the arguments of $\mathcal{H},$ but of course $\mathcal{H}$ depends on the underlying solution
$\nabla u$ and $m$ rather than the linearized quantities $\nabla v$ and $\mu.$
We intialize the iteration with $v^{0}=0$ and $\mu^{0}=0.$  Note that $v^{0}$ and $\mu^{0}$ are (trivially) infinitely smooth,
and and $v^{n+1}$ and $\mu^{n+1}$ satisfy forced heat equations with infinitely smooth data and forcing.  Therefore
$v^{n+1}$ and $\mu^{n+1}$ exist and are infinitely smooth.  Note that while of course the solutions depend on both
regularization parameters $n$ and $\varepsilon,$ we suppress the dependence on $\varepsilon$ for the time being.

We now make energy estimates for $v^{n+1}$ and $\mu^{n+1}.$  Define $E_{v}$ and $E_{\mu}$ to be
\begin{equation}\nonumber
E_{v}=\frac{1}{2}\|v^{n+1}\|_{H^{s}}^{2}=\frac{1}{2}\int_{\mathbb{T}^{d}}(\Lambda^{s}v^{n+1})^{2}\ dx,
\end{equation}
\begin{equation}\nonumber
E_{\mu}=\frac{1}{2}\|\mu^{n+1}\|_{H^{s-1}}^{2}=\frac{1}{2}\int_{\mathbb{T}^{d}}(\Lambda^{s-1}\mu^{n+1})^{2}\ dx.
\end{equation}

Then the time derivative of $E_{v}$ satisfies
\begin{multline}\nonumber
\frac{dE_{v}}{dt}=\int_{\mathbb{T}^{d}}(\Lambda^{s}v^{n+1})(\Lambda^{s}\partial_{t}v^{n+1})\ dx
=-\int_{\mathbb{T}^{d}}(\Lambda^{s}v^{n+1})(\Lambda^{s}\Delta v^{n+1})\ dx
\\
+\int_{\mathbb{T}^{d}}(\Lambda^{s}v^{n+1})(\mathcal{J}_{\varepsilon}\Lambda^{s}(D_p\mathcal{H}\cdot\nabla v^{n}))\ dx
+\int_{\mathbb{T}^{d}}(\Lambda^{s}v^{n+1})(\mathcal{J}_{\varepsilon}\Lambda^{s}(\partial_q\mathcal{H}\mu^{n})\ dx
=I_{v}+II_{v}+III_{v}.
\end{multline}
We integrate this in time from time $t$ to time $T,$ arriving at (after slight rearranging of terms)
\begin{equation}\nonumber
\|v^{n+1}(t,\cdot)\|_{H^{s}}^{2}=\|v^{n+1}(T,\cdot)\|_{H^{s}}^{2}-2\int_{t}^{T}I_{v}\ d\tau
-2\int_{t}^{T}II_{v}\ d\tau - 2\int_{t}^{T}III_{v}\ d\tau.
\end{equation}

The time derivative of $E_{\mu}$ satisfies
\begin{multline}\nonumber
\frac{dE_{\mu}}{dt}=\int_{\mathbb{T}^{d}}(\Lambda^{s-1}\mu^{n+1})(\Lambda^{s-1}\partial_{t}\mu^{n+1})\ dx
=\int_{\mathbb{T}^{d}}(\Lambda^{s-1}\mu^{n+1})(\Lambda^{s-1}\Delta\mu^{n+1})\ dx
\\
+\int_{\mathbb{T}^{d}}(\Lambda^{s-1}\mu^{n+1})(\Lambda^{s-1}\mathcal{J}_{\varepsilon}\mathrm{div}
(\mu^{n}D_{p}\mathcal{H}))\ dx
+\int_{\mathbb{T}^{d}}(\Lambda^{s-1}\mu^{n+1})(\Lambda^{s-1}\mathcal{J}_{\varepsilon}\mathrm{div}
(mD_{pp}^{2}\mathcal{H}\cdot\nabla v^{n}))\ dx
\\
+\int_{\mathbb{T}^{d}}(\Lambda^{s-1}\mu^{n+1})(\Lambda^{s-1}\mathcal{J}_{\varepsilon}\mathrm{div}
(mD_p\partial_q\cH\mu^{n}))\ dx
=I_{\mu}+II_{\mu}+III_{\mu}+IV_{\mu}.
\end{multline}
We integrate this in time from time zero to time $t,$ finding
\begin{equation}\nonumber
\|\mu^{n+1}(t,\cdot)\|_{H^{s-1}}^{2}=\|\mu^{n+1}(0,\cdot)\|_{H^{s-1}}^{2}
+2\int_{0}^{t}I_{\mu}\ d\tau +2\int_{0}^{t}II_{\mu}\ d\tau +2\int_{0}^{t}III_{\mu}\ d\tau + 2\int_{0}^{t}IV_{\mu}\ d\tau.
\end{equation}

We begin estimating with the terms involving $I_{v},$ $II_{v},$ and $III_{v}.$  For $I_{v}$ we integrate by parts:
\begin{equation}\nonumber
I_{v}=\int_{\mathbb{T}_{d}}(\Lambda^{s}\nabla v^{n+1})^{2}\ dx,
\end{equation}
and integrating in time we have
\begin{equation}\nonumber
-\int_{t}^{T}I_{v}=-\int_{t}^{T}\|\nabla v^{n+1}\|_{H^{s}}^{2}\ d\tau.
\end{equation}
For the term with $II_{v}$ we use Young's inequality with positive parameter $\sigma_{1}$ and \eqref{estim:Hr-Hs} to obtain
\begin{equation}\nonumber
|II_{v}|\leq\frac{\sigma_{1}}{2}\|v^{n+1}\|_{H^{s}}^{2}+\frac{1}{2\sigma_{1}}\|D_{p}\mathcal{H}\cdot\nabla v^{n}\|_{H^{s}}^{2}
\leq\frac{\sigma_{1}}{2}\|v^{n+1}\|_{H^{s}}^{2}+\frac{c}{\sigma_{1}}\|\nabla v^{n}\|_{H^{s}}^{2}.
\end{equation}
Integrating in time and making some elementary manipulations, this becomes
\begin{equation}\nonumber
-\int_{t}^{T}II_{v}\ d\tau\leq\frac{\sigma_{1}T}{2}\sup_{t\in[0,T]}\|v^{n+1}\|_{H^{s}}^{2}
+\frac{c}{\sigma_{1}}\int_{t}^{T}\|\nabla v^{n}\|_{H^{s}}^{2}\ d\tau.
\end{equation}
We treat the term with $III_{v}$ similarly:
\begin{equation}\nonumber
|III_{v}|\leq\frac{\sigma_{2}}{2}\|v^{n+1}\|_{H^{s}}^{2}+\frac{1}{2\sigma_{2}}\|\partial_q\mathcal{H}\mu^{n}\|_{H^{s}}^{2}
\leq\frac{\sigma_{2}}{2}\|v^{n+1}\|_{H^{s}}^{2}+\frac{c}{\sigma_{2}}\|\mu^{n}\|_{H^{s}}^{2},
\end{equation}
and integrating in time,
\begin{equation}\nonumber
-\int_{0}^{T}III_{v}\ d\tau\leq\frac{\sigma_{2}T}{2}\sup_{t\in[0,T]}\|v^{n+1}\|_{H^{s}}^{2}+\frac{c}{\sigma_{2}}
\int_{t}^{T}\|\mu^{n}\|_{H^{s}}^{2}\ d\tau.
\end{equation}

We next turn to the terms coming from $E_{\mu}.$  We begin with the $I_{\mu}$ term, integrating by parts:
\begin{equation}\nonumber
I_{\mu}=-\int_{\mathbb{T}^{d}}(\Lambda^{s-1}\nabla\mu^{n+1})^{2}\ dx.
\end{equation}
Integrating this with respect to time, we find
\begin{equation}\nonumber
\int_{0}^{t}I_{\mu}\ d\tau = - \int_{0}^{t}\|\nabla\mu^{n+1}\|_{H^{s-1}}^{2}\ d\tau.
\end{equation}
For $II_{\mu},$ we again use Young's inequality, now with positive parameter $\sigma_{3}:$
\begin{equation}\nonumber
|II_{\mu}|\leq\frac{\sigma_{3}}{2}\|\mu^{n+1}\|_{H^{s-1}}^{2}+\frac{1}{2\sigma_{3}}\|\mu^{n}D_{p}\mathcal{H}\|_{H^{s}}^{2}
\leq\frac{\sigma_{3}}{2}\|\mu^{n+1}\|_{H^{s-1}}^{2}+\frac{c}{\sigma_{3}}\|\mu^{n}\|_{H^{s}}^{2}.
\end{equation}
We use the inequality
\begin{equation}\label{dealingWithGradient}
\|\mu^{n}\|_{H^{s}}^{2}\leq \|\mu^{n}\|_{H^{s-1}}^{2}+c\|\nabla\mu^{n}\|_{H^{s-1}}^{2},
\end{equation}
and we integrate in time, finding
\begin{equation}\nonumber
\int_{0}^{t}|II_{\mu}|\ d\tau\leq\frac{\sigma_{3}T}{2}\sup_{t\in[0,T]}\|\mu^{n+1}\|_{H^{s-1}}^{2}
+\frac{cT}{\sigma_{3}}\sup_{t\in[0,T]}\|\mu^{n}\|_{H^{s-1}}^{2}
+\frac{c}{\sigma_{3}}\int_{0}^{t}\|\nabla\mu^{n}\|_{H^{s-1}}^{2}\ d\tau.
\end{equation}
We treat $III_{\mu}$ similarly:
\begin{equation}\nonumber
|III_{\mu}|\leq\frac{\sigma_{4}}{2}\|\mu^{n+1}\|_{H^{s-1}}^{2}
+\frac{1}{2\sigma_{4}}\|mD_{pp}^{2}\mathcal{H}\cdot\nabla v^{n}\|_{H^{s}}^{2}
\leq\frac{\sigma_{4}}{2}\|\mu^{n+1}\|_{H^{s-1}}^{2}+\frac{c}{\sigma_{4}}\|\nabla v^{n}\|_{H^{s}}^{2}.
\end{equation}
Integrated in time, this becomes
\begin{equation}\nonumber
\int_{0}^{t}|III_{\mu}|\ d\tau\leq\frac{\sigma_{4}T}{2}\sup_{t\in[0,T]}\|\mu^{n+1}\|_{H^{s-1}}^{2}
+\frac{c}{\sigma_{4}}\int_{0}^{t}\|\nabla v^{n}\|_{H^{s}}^{2}\ d\tau.
\end{equation}
Finally, we work with $IV_{\mu}$ along similar lines,
\begin{equation}\nonumber
|IV_{\mu}|\leq\frac{\sigma_{5}}{2}\|\mu^{n+1}\|_{H^{s-1}}^{2}+\frac{1}{2\sigma_{5}}\|mD_p\partial_q\mathcal{H}\mu^{n}\|_{H^{s}}^{2}
\leq\frac{\sigma_{5}}{2}\|\mu^{n+1}\|_{H^{s-1}}^{2}+\frac{c}{\sigma_{5}}\|\mu^{n}\|_{H^{s}}^{2}.
\end{equation}
We again use \eqref{dealingWithGradient}, and integrate in time to find
\begin{equation}\nonumber
\int_{0}^{t}|IV_{\mu}|\ d\tau\leq\frac{\sigma_{5}T}{2}\sup_{t\in[0,T]}\|\mu^{n+1}\|_{H^{s-1}}^{2}
+\frac{cT}{\sigma_{5}}\sup_{t\in[0,T]}\|\mu^{n}\|_{H^{s-1}}^{2}
+\frac{c}{\sigma_{5}}\int_{0}^{t}\|\nabla \mu^{n}\|_{H^{s-1}}^{2}\ d\tau.
\end{equation}

We summarize what we have concluded so far, and we use \eqref{payoffAssumption}:
\begin{multline}\label{almostVBound}
\|v^{n+1}(t,\cdot)\|_{H^{s}}^{2}+\int_{t}^{T}\|\nabla v^{n+1}\|_{H^{s}}^{2}\ d\tau
\leq \kappa\sup_{t\in[0,T]}\|\mu^{n+1}\|_{H^{s-1}}^{2}
+(\sigma_{1}+\sigma_{2})T\sup_{t\in[0,T]}\|v^{n+1}\|_{H^{s}}^{2}
\\
+\frac{cT}{\sigma_{2}}\sup_{t\in[0,T]}\|\mu^{n}\|_{H^{s-1}}^{2}
+\frac{c}{\sigma_{1}}\int_{0}^{T}\|\nabla v^{n}\|_{H^{s}}^{2}\ d\tau
+\frac{c}{\sigma_{2}}\int_{0}^{T}\|\nabla\mu^{n}\|_{H^{s-1}}^{2}\ d\tau,
\end{multline}
\begin{multline}\label{almostMuBound}
\|\mu^{n+1}(t,\cdot)\|_{H^{s-1}}^{2}+\int_{0}^{t}\|\nabla\mu^{n+1}\|_{H^{s-1}}^{2}\ d\tau
\leq\|\mu_{0}\|_{H^{s-1}}^{2}+(\sigma_{3}+\sigma_{4}+\sigma_{5})T\sup_{t\in[0,T]}\|\mu^{n+1}\|_{H^{s-1}}^{2}
\\
+cT\left(\frac{1}{\sigma_{3}}+\frac{1}{\sigma_{5}}\right)\sup_{t\in[0,T]}\|\mu^{n}\|_{H^{s-1}}^{2}
+c\left(\frac{1}{\sigma_{3}}+\frac{1}{\sigma_{5}}\right)\int_{0}^{T}\|\nabla\mu^{n}\|_{H^{s-1}}^{2}\ d\tau
+\frac{c}{\sigma_{4}}\int_{0}^{T}\|\nabla v^{n}\|_{H^{s}}^{2}\ d\tau.
\end{multline}
We take $\sigma_{1}=\sigma_{2}=\frac{1}{8T},$ and we take $\sigma_{3}=\sigma_{4}=\sigma_{5}=\frac{1}{12T}.$
We multiply \eqref{almostVBound} by $\frac{1}{8\kappa},$ and we add the result to \eqref{almostMuBound}, taking into 
account these choices of our parameters.  Neglecting the integral terms on the left-hand sides of \eqref{almostVBound} 
and \eqref{almostMuBound}, we are able to take the supremum in time.  The result after some slight rearranging is
\begin{multline}\label{reallyAlmostVBound}
\frac{3}{32\kappa}\sup_{t\in[0,T]}\|v^{n+1}\|_{H^{s}}^{2}+\frac{5}{8}\sup_{t\in[0,T]}\|\mu^{n+1}\|_{H^{s-1}}^{2}
\\
\leq \|\mu_{0}\|_{H^{s-1}}^{2}
+cT^{2}\sup_{t\in[0,T]}\|\mu^{n}\|_{H^{s-1}}^{2}
+cT\int_{0}^{T}\|\nabla v^{n}\|_{H^{s}}^{2}+\|\nabla\mu^{n}\|_{H^{s-1}}^{2}\ d\tau.
\end{multline}
If we instead neglect the other terms on the left-hand side of \eqref{almostVBound} and \eqref{almostMuBound},
we can instead find that
\begin{multline}\label{reallyAlmostMuBound}
\frac{1}{8\kappa}\int_{0}^{T}\|\nabla v^{n+1}\|_{H^{s}}^{2}\ d\tau
+\int_{0}^{T}\|\nabla\mu^{n+1}\|_{H^{s-1}}^{2}\ d\tau
\leq \frac{1}{32\kappa}\sup_{t\in[0,T]}\|v^{n+1}\|_{H^{s}}^{2}
\\
+\frac{3}{8}\sup_{t\in[0,T]}\|\mu^{n+1}\|_{H^{s-1}}^{2}+
\|\mu_{0}\|_{H^{s-1}}^{2}
+cT^{2}\sup_{t\in[0,T]}\|\mu^{n}\|_{H^{s-1}}^{2}
+cT\int_{0}^{T}\|\nabla v^{n}\|_{H^{s}}^{2}+\|\nabla\mu^{n}\|_{H^{s-1}}^{2}\ d\tau.
\end{multline}
Adding \eqref{reallyAlmostVBound} and \eqref{reallyAlmostMuBound}, and slightly rearranging, we conclude
\begin{multline}\label{readyForInduction}
\frac{1}{16\kappa}\sup_{t\in[0,T]}\|v^{n+1}\|_{H^{s}}^{2}+\frac{1}{4}\sup_{t\in[0,T]}\|\mu^{n+1}\|_{H^{s-1}}^{2}
+\frac{1}{8\kappa}\int_{0}^{T}\|\nabla v^{n+1}\|_{H^{s}}^{2}\ d\tau + \int_{0}^{T}\|\nabla\mu^{n+1}\|_{H^{s-1}}^{2}\ d\tau
\\
\leq 2\|\mu_{0}\|_{H^{s-1}}^{2}+cT^{2}\sup_{t\in[0,T]}\|\mu^{n}\|_{H^{s-1}}^{2}
+cT\int_{0}^{T}\|\nabla v^{n}\|_{H^{s}}^{2}+\|\nabla\mu^{n}\|_{H^{s-1}}^{2}\ d\tau.
\end{multline}

We can now perform our induction.  We wish to prove that for all $n,$ we have
\begin{equation}\label{inductiveHypothesis}
\frac{1}{16\kappa}\sup_{t\in[0,T]}\|v^{n}\|_{H^{s}}^{2}+\frac{1}{4}\sup_{t\in[0,T]}\|\mu^{n}\|_{H^{s-1}}^{2}
+\frac{1}{8\kappa}\int_{0}^{T}\|\nabla v^{n}\|_{H^{s}}^{2}\ d\tau + \int_{0}^{T}\|\nabla\mu^{n}\|_{H^{s-1}}^{2}\ d\tau
\leq 3\|\mu_{0}\|_{H^{s-1}}^{2}.
\end{equation}
Clearly \eqref{inductiveHypothesis} holds when $n=0$ since $v^{0}=\mu^{0}=0.$  We assume \eqref{inductiveHypothesis}
holds for some value of $n.$  Then from \eqref{readyForInduction}, we see that for $T$ sufficiently small, 
\eqref{inductiveHypothesis} holds for the value $n+1.$  This restriction on the size of $T$ is independent of $n.$
We therefore have established \eqref{inductiveHypothesis} for all values of $n,$ and this is a uniform bound
on our iterates.

We now pass to the limit.  Recall that our solutions depend on two parameters, since there is a mollification parameter
$\varepsilon$ which we have suppressed, in addition to the iteration parameter $n.$  
Note that the above bound \eqref{inductiveHypothesis} which is
uniform with respect to $n$ is also uniform with respect to $\varepsilon.$
We will now take $\varepsilon=\frac{1}{n}$ and simply pass to the limit as $n\rightarrow\infty.$  
Our sequence $(v^{n},\mu^{n})$ is uniformly bounded in $L^\infty([0,T];H^{s}(\T^d))\times L^\infty([0,T];H^{s-1}(\T^d)),$ with $s>3+\frac{d}{2}.$  
By Sobolev imbedding this implies that $\|\nabla v^{n}\|_{L^{\infty}}$, $\|\nabla\mu^{n}\|_{L^{\infty}}$ and $\|\Delta v^{n}\|_{L^{\infty}}$, $\|\Delta\mu^{n}\|_{L^{\infty}}$ are also uniformly bounded.  Inspection of the iterated equations \eqref{iteratedV} and \eqref{iteratedMu} together with
the uniform bound and Sobolev embedding further implies that $\|\partial_{t}v^{n}\|_{L^{\infty}}$
and $\|\partial_{t}\mu^{n}\|_{L^{\infty}}$ are uniformly bounded.  We conclude that $v^{n}$ forms a bounded
equicontinuous family, as does $\mu^{n}.$  By the Arzel\`a-Ascoli theorem, there is a uniformly convergent
subsequence, which we do not relabel.  We thus have a limit, $(v,\mu).$

Our sequence $(v^{n},\mu^{n})$ converges to $(v,\mu)$ in $(C([0,T]\times\mathbb{T}^{d}))^{2},$ and thus
also in $C([0,T];L^{2}(\mathbb{T}^{d})\times L^{2}(\mathbb{T}^{d})).$  
With the uniform bound \eqref{inductiveHypothesis} and the
Sobolev interpolation inequality \eqref{sobolevInterpolation}, we may then conclude that the convergence
actually holds in $C([0,T];H^{s'}(\mathbb{T}^{d})\times H^{s'-1}(\mathbb{T}^{d})),$ for any $s'<s.$

Now we may conclude that $(v,\mu)$ satisfies the appropriate system, which is \eqref{eq:linearized}.
Integrating \eqref{iteratedV} with respect to time, we have
\begin{equation}\label{iteratedVIntegratedInTime}
v^{n+1}(t,\cdot)=\mathcal{J}_{1/n}\left(\frac{\delta G}{\delta m}\mu^{n}(T,\cdot)\right)
+\int_{t}^{T}\Big(\Delta v^{n+1} -\mathcal{J}_{1/n}(D_{p}\mathcal{H}\cdot\nabla v^{n})
-\mathcal{J}_{1/n}(\partial_q\mathcal{H}\mu^{n})\Big)\ d\tau.
\end{equation}
(Recall that we have set $\varepsilon=1/n.$)  We have established uniform convergence in sufficiently regular
spaces to be able to take the limit in every term in \eqref{iteratedVIntegratedInTime}.  We therefore have
\begin{equation}\nonumber
v(t,\cdot)=\frac{\delta G}{\delta m}\mu(T,\cdot)+\int_{t}^{T}\Big(\Delta v -D_{p}\mathcal{H}\cdot\nabla v -\partial_q\mathcal{H}\mu\Big)\ d\tau.
\end{equation}
We can treat $\mu$ analagously; we conclude that \eqref{eq:linearized} is satisfied by $(v,\mu).$

Finally, we note that since the unit ball of a Hilbert space is weakly compact, the sequences $v^{n}(t,\cdot)$ and
$\mu^{n}(t,\cdot)$ have weak limits in $H^{s}(\T^d)\times H^{s-1}(\T^d)$ at every time.  This limit must be the same as the
limit previously found, therefore we may conclude 
$(v,\mu)\in L^{\infty}([0,T];H^{s}(\mathbb{T}^{d})\times H^{s-1}(\mathbb{T}^{d}))$ as well. 
\end{proof}

\begin{remark}
(1) By optimizing the time horizon in \eqref{reallyAlmostMuBound} and \eqref{reallyAlmostVBound} in the above proof, we find that 
\begin{align*}
T_*\ge c_0\frac{1}{c\kappa},
\end{align*}
where $c_0>0$ is a universal constant, $\kappa$ is constant appearing in \eqref{payoffAssumption}, while 
\small
$$c=c\left(\|D_p\cH(t,\cdot,m,\nabla u)\|_{H^s},\|\partial_q\cH(t,\cdot,m,\nabla u)\|_{H^s},\|mD^2_{pp}\cH(t,\cdot,m,\nabla u)\|_{H^s},\|mD_p\partial_q\cH(t,\cdot,m,\nabla u)\|_{H^s}\right).$$
\normalsize

(2) In the same time, we have also that  $\|v\|_{L^\infty(H^s)}\le \tilde c_0\|\mu_0\|_{H^{s-1}},$ for some universal constant $\tilde c_0>0$.
\end{remark}

\begin{corollary}\label{cor:v_regular}
Under the assumption that $ s>\left\lceil\frac{d+5}{2}\right\rceil$, by Sobolev embedding theorem, the previous theorem yields $v\in C([0,T]; C^{2,\alpha}(\T^d))$ for some $\a\in[0,1)$.
\end{corollary}

\section{More estimates on the solution of the linearized system}\label{sec:step2}

To get the desired regularity on the kernel $K$ with respect to the $y$ variable, we will use a Riesz type representation
theorem.  To conclude that $K$ is in $H^{s}$ we will use the following theorem, which shows that
$v$ is related to a continuous linear operator on $H^{-s}.$

\begin{theorem}\label{thm:v_H-s_bound}
Let $v$ and $\mu_{0}$ be as in Theorem \ref{step1Theorem}. There exists $T_{**}>0$ with $T_{**}\le T_*$ and $c>0$ such that if $T\in(0,T_{**}),$ then
for $s$ satisfying \eqref{hyp:s} 
\begin{align}\label{estim:thm-H-s}
\sup_{t\in[0,T]}\|v(t,\cdot)\|_{H^{-s}}\leq c\|\mu_{0}\|_{H^{-s-1}}.
\end{align}
\end{theorem}

\begin{proof}
We define an energy, beginning with the piece
\begin{equation}\nonumber
E_{v}=\frac{1}{2}\|v\|_{H^{-s}}^{2}=\int_{\mathbb{T}^{d}}\left(\Lambda^{-s}v\right)^{2}\ dx.
\end{equation}
We take the time derivative of $E_{v},$ finding
\begin{equation}\nonumber
\frac{dE_{v}}{dt}=\int_{\mathbb{T}^{d}}\left(\Lambda^{-s}v\right)\left(\Lambda^{-s}\partial_{t}v\right)\ dx.
\end{equation}
We now substitute from the $\partial_{t}v$ equation in \eqref{eq:linearized}, finding
\begin{align*}\nonumber
\frac{dE_{v}}{dt}=&-\int_{\mathbb{T}^{d}}\left(\Lambda^{-s}v\right)\left(\Lambda^{-s}\Delta v\right) dx
+\int_{\mathbb{T}^{d}}\left(\Lambda^{-s}v\right)\left(\Lambda^{-s}(D_{p}\mathcal{H}\cdot\nabla v)\right) dx\\
+&\int_{\mathbb{T}^{d}}\left(\Lambda^{-s}v\right)\left(\Lambda^{-s}(\partial_q\mathcal{H}\mu)\right) dx
=I_{v}+II_{v}+III_{v},
\end{align*}
where we have suppressed the arguments in $\mathcal{H}.$  
We then integrate in time from time $t$ to time $T,$ finding 
\begin{equation}\nonumber
E_{v}(t)=E_{v}(T)-\int_{t}^{T}I_{v}\ d\tau - \int_{t}^{T}II_{v}\ d\tau - \int_{t}^{T}III_{v}\ d\tau.
\end{equation}

We integrate by parts with respect to the spatial variables in the term involving $I_{v}:$
\begin{equation}\label{eq:july04.2020.1}
-\int_{t}^{T}I_{v}\ d\tau=-\int_{t}^{T}\int_{\mathbb{T}^{d}}(\Lambda^{-s}\nabla v)^{2}\ dxd\tau
=-\int_{t}^{T}\|\Lambda^{-s}\nabla v\|_{L^{2}}^{2}\ d\tau.
\end{equation}
This term gives us a parabolic gain of regularity and will be useful as we continue our estimates.

For the term involving $II_{v},$ we begin by estimating with Cauchy-Schwarz, and we recall from \eqref{-1norm} that
for any $\phi,$ we have $\|\phi\|_{H^{-s}}=\|\Lambda^{-s}\phi\|_{L^{2}}:$
\begin{equation}\nonumber
\left|II_{v}\right|\leq\|v\|_{H^{-s}}\|D_{p}\mathcal{H}\cdot\nabla v\|_{H^{-s}}.
\end{equation}
We then use \eqref{estim:H-s}, finding
\begin{equation}\nonumber
\left|II_{v}\right|\leq c\|v\|_{H^{-s}}\|D_{p}\mathcal{H}\|_{H^{s}}\|\nabla v\|_{H^{-s}}.
\end{equation}
We use Young's inequality with positive parameter $\sigma_{1}:$
\begin{equation}\nonumber
\left|II_{v}\right|\leq \frac{c}{\sigma_{1}}\|v\|_{H^{-s}}^{2}\|D_{p}\mathcal{H}\|_{H^{s}}^{2}
+\frac{\sigma_{1}}{2}\|\Lambda^{-s}\nabla v\|_{L^{2}}^{2}.
\end{equation}
We then integrate in time:
\begin{equation}\label{eq:july04.2020.2}
-\int_{t}^{T}II_{v}\ d\tau\leq \frac{c}{\sigma_{1}}\int_{t}^{T}\|v\|_{H^{-s}}^{2}\|D_{p}\mathcal{H}\|_{H^{s}}^{2} d\tau
+\frac{\sigma_{1}}{2}\int_{t}^{T}\|\Lambda^{-s}\nabla v\|_{L^{2}}^{2} d\tau.
\end{equation}

We now turn to $III_{v},$ estimating similarly to before:
\begin{equation}\nonumber
|III_{v}|\leq \|v\|_{H^{-s}}\|\partial_q\mathcal{H}\mu\|_{H^{-s}}\leq c\|v\|_{H^{-s}}\|\partial_q\mathcal{H}\|_{H^{s}}
\|\mu\|_{H^{-s}}.
\end{equation}
We next use Young's inequality with positive parameter $\sigma_{2}:$
\begin{equation}\nonumber
|III_{v}|\leq \frac{c}{\sigma_{2}}\|v\|_{H^{-s}}^{2}\|\partial_q\mathcal{H}\|_{H^{s}}^{2}+\frac{\sigma_{2}}{2}\|\mu\|_{H^{-s}}^{2}.
\end{equation}
We integrate this in time:
\begin{equation}\label{eq:july04.2020.3}
-\int_{t}^{T}III_{v}\ d\tau\leq\frac{c}{\sigma_{2}}\int_{t}^{T}\|v\|_{H^{-s}}^{2}\|\partial_q\mathcal{H}\|_{H^{s}}^{2} d\tau
+\frac{\sigma_{2}}{2}\int_{t}^{T}\|\mu\|_{H^{-s}}^{2} d\tau.
\end{equation}

We next define the other piece of the energy, $E_{\mu},$ which is
\begin{equation}\nonumber
E_{\mu}=\frac{1}{2}\|\mu\|_{H^{-s-1}}^{2}=\frac12\int_{\mathbb{T}^{d}}\left(\Lambda^{-s-1}\mu\right)^{2} dx.
\end{equation}
The time derivative of this is
\begin{equation}\nonumber
\frac{dE_{\mu}}{dt}=\int_{\mathbb{T}^{d}}\left(\Lambda^{-s-1}\mu\right)\left(\Lambda^{-s-1}\partial_{t}\mu\right) dx.
\end{equation}
We substitute from the $\partial_{t}\mu$ equation in \eqref{eq:linearized}, finding
\begin{multline}\nonumber
\frac{dE_{\mu}}{dt}=\int_{\mathbb{T}^{d}}\left(\Lambda^{-s-1}\mu\right)\left(\Lambda^{-s-1}\Delta\mu\right) dx
-\int_{\mathbb{T}^{d}}\left(\Lambda^{-s-1}\mu\right)\left(\Lambda^{-s-1}\nabla\cdot(\mu D_{p}\mathcal{H})\right) dx
\\
-\int_{\mathbb{T}^{d}}\left(\Lambda^{-s-1}\mu\right)\left(\Lambda^{-s-1}\nabla\cdot(m(D_{pp}^{2}\mathcal{H})\nabla v)\right) dx
-\int_{\mathbb{T}^{d}}\left(\Lambda^{-s-1}\mu\right)\left(\Lambda^{-s-1}\nabla\cdot(m(D_p\partial_q\mathcal{H})\mu)\right) dx
\\
=I_{\mu}+II_{\mu}+III_{\mu}+IV_{\mu}.
\end{multline}
We integrate in time from time zero until time $t,$
\begin{equation}\nonumber
E_{\mu}(t)=E_{\mu}(0)+\int_{0}^{t}I_{\mu}\ d\tau + \int_{0}^{t}II_{\mu}\ dt
+\int_{0}^{t}III_{\mu}\ d\tau + \int_{0}^{t}IV_{\mu}\ d\tau.
\end{equation}
We integrate the term involving $I_{\mu}$ by parts, to find
\begin{equation}\label{eq:july04.2020.4.13}
\int_{0}^{t}I_{\mu}\ d\tau=-\int_{0}^{t}\int_{\mathbb{T}^{d}}\left(\Lambda^{-s-1}\nabla\mu\right)^{2}\ dxd\tau.
\end{equation}
This is a parabolic term which will be helpful to us, as the term $I_{v}$ was as well.

We next note that for any $f\in H^{-s},$ we have $\|\Lambda^{-s-1}\nabla f\|_{L^{2}}\leq c\|f\|_{H^{-s}}.$
The $H^{-s}$-norm of $f,$ though, is not equivalent to the $L^{2}$-norm of $\Lambda^{-s-1}\nabla f,$ because of the $k=0$
mode.  We instead will use the inequality
\begin{equation}\label{notQuiteEquivalent}
\|f\|_{H^{-s}}^{2}\leq c\|f\|_{H^{-s-1}}^{2}+c\|\Lambda^{-s-1}\nabla f\|_{L^{2}}^{2}.
\end{equation}

We begin to estimate the rest of the terms, starting with the $II_{\mu}$ term,
\begin{equation}\nonumber
|II_{\mu}|\leq c\|\Lambda^{-s-1}\mu\|_{L^{2}}\|D_{p}\mathcal{H}\|_{H^{s}}\|\mu\|_{H^{-s}}.
\end{equation}
We use Young's inequality with positive parameter $\sigma_{3},$
\begin{equation}\nonumber
|II_{\mu}|\leq \frac{c}{\sigma_{3}}\|D_{p}\mathcal{H}\|_{H^{s}}^{2}\|\mu\|_{H^{-s-1}}^{2}
+\frac{\sigma_{3}}{2}\|\mu\|_{H^{-s}}^{2}.
\end{equation}
We finally use \eqref{notQuiteEquivalent} with this, yielding
\begin{equation}\nonumber
|II_{\mu}|\leq \frac{c}{\sigma_{3}}\|D_{p}\mathcal{H}\|_{H^{s}}^{2}\|\mu\|_{H^{-s-1}}^{2}
+\frac{c\sigma_{3}}{2}\|\mu\|_{H^{-s-1}}^{2}+\frac{c\sigma_{3}}{2}\|\Lambda^{-s-1}\nabla\mu\|_{L^{2}}^{2}.
\end{equation}
We integrate in time, finding
\begin{equation}\label{eq:july04.2020.4.15}
\int_{0}^{t}|II_{\mu}| d\tau \leq 
\frac{c}{\sigma_{3}}\int_{0}^{t}\|D_{p}\mathcal{H}\|_{H^{s}}^{2}\|\mu\|_{H^{-s-1}}^{2}\ d\tau
+\frac{c\sigma_{3}}{2}\int_{0}^{t}\|\mu\|_{H^{-s-1}}^{2}\ d\tau
+\frac{c\sigma_{3}}{2}\int_{0}^{t}\|\Lambda^{-s-1}\nabla\mu\|_{L^{2}}^{2}\ d\tau.
\end{equation}

We next consider $III_{\mu}:$ 
\begin{equation}
|III_{\mu}|\leq c\|\Lambda^{-s-1}\mu\|_{L^{2}}\|mD_{pp}^{2}\mathcal{H}\|_{H^{s}}\|\nabla v\|_{H^{-s}}.
\end{equation}
We use Young's inequality with positive parameter $\sigma_{4}:$
\begin{equation}\nonumber
|III_{\mu}|\leq \frac{c}{\sigma_{4}}\|\Lambda^{-s-1}\mu\|_{L^{2}}^{2}\|mD_{pp}^{2}\mathcal{H}\|_{H^{s}}^{2}
+\frac{\sigma_{4}}{2}\|\Lambda^{-s}\nabla v\|_{L^{2}}^{2}.
\end{equation}
Integrated in time, this becomes
\begin{equation}\label{eq:july04.2020.4.17}
\int_{0}^{t}|III_{\mu}| d\tau \leq 
\frac{c}{\sigma_{4}}\int_{0}^{t}\|\Lambda^{-s-1}\mu\|_{L^{2}}^{2}\|mD_{pp}^{2}\mathcal{H}\|_{H^{s}}^{2}\ d\tau
+\frac{\sigma_{4}}{2}\int_{0}^{t}\|\Lambda^{-s}\nabla v\|_{L^{2}}^{2}\ d\tau.
\end{equation}

Our final term to estimate is $IV_{\mu},$ and we use Cauchy-Schwarz to bound this as
\begin{equation}\nonumber
|IV_{\mu}|\leq c\|\Lambda^{-s-1}\mu\|_{L^{2}}\|mD_p\partial_q\mathcal{H}\|_{H^{s}}\|\mu\|_{H^{-s}}.
\end{equation}
We use Young's inequality with positive parameter $\sigma_{5}:$
\begin{equation}\nonumber
|IV_{\mu}|\leq \frac{c}{\sigma_{5}}\|\Lambda^{-s-1}\mu\|_{L^{2}}^{2}\|mD_p\partial_q\mathcal{H}\|_{H^{s}}^{2}
+\frac{\sigma_{5}}{2}\|\mu\|_{H^{-s}}^{2}.
\end{equation}
Next we use \eqref{notQuiteEquivalent}, yielding
\begin{equation}\nonumber
|IV_{\mu}|\leq \frac{c}{\sigma_{5}}\|\Lambda^{-s-1}\mu\|_{L^{2}}^{2}\|mD_p\partial_q\mathcal{H}\|_{H^{s}}^{2}
+\frac{c\sigma_{5}}{2}\|\mu\|_{H^{-s-1}}^{2}+\frac{c\sigma_{5}}{2}\|\Lambda^{-s-1}\nabla\mu\|_{L^{2}}^{2}.
\end{equation}
Integrated in time, this is
\begin{align}
& \int_{0}^{t} |IV_{\mu}|\ d\tau \nonumber \\
\leq&  \frac{c}{\sigma_{5}}\int_{0}^{t}\|\Lambda^{-s-1}\mu\|_{L^{2}}^{2}\|mD_p\partial_q\mathcal{H}\|_{H^{s}}^{2}\ d\tau
+\frac{c\sigma_{5}}{2} \Big(\int_{0}^{t}\|\mu\|_{H^{-s-1}}^{2}\ d\tau+ \int_{0}^{t}\|\Lambda^{-s-1}\nabla\mu\|_{L^{2}}^{2}\ d\tau\Big). \label{eq:july04.2020.4.18}
\end{align}

We choose our parameters from Young's inequality as follows:
\begin{equation}\nonumber
\frac{\sigma_{1}}{2}=\frac{1}{8},\qquad \frac{\sigma_{2}}{2}=\frac{1}{12},\qquad \frac{c\sigma_{3}}{2}=\frac{1}{12},
\qquad \frac{\sigma_{4}}{2}=\frac{1}{8\kappa},\qquad \frac{c\sigma_{5}}{2}=\frac{1}{12}.
\end{equation}

Then, our findings from our energy estimates can be summarized in the following inequalities.  First by \eqref{eq:july04.2020.1}-\eqref{eq:july04.2020.3} we have
\begin{align*}
& E_{v}(t)+ \frac{7}{8}\int_{t}^{T}\|\Lambda^{-s}\nabla v\|_{L^{2}}^{2}\ d\tau \\
\leq & E_{v}(T)+4c \int_{t}^{T}\|v\|_{H^{-s}}^{2}\|D_{p}\mathcal{H}\|_{H^{s}}^{2} d\tau +6c \int_{t}^{T}\|v\|_{H^{-s}}^{2}\|\partial_q\mathcal{H}\|_{H^{s}}^{2} d\tau+ \frac{1}{12}\int_{0}^{T}\|\mu\|_{H^{-s}}^{2}\ d\tau
\end{align*}
Increasing the value of $c$ if necessary  we have 
\[
E_{v}(t)+ \frac{7}{8}\int_{t}^{T}\|\Lambda^{-s}\nabla v\|_{L^{2}}^{2}\ d\tau \leq E_{v}(T)
+c T\|v\|_{H^{-s}}^{2} + \frac{1}{12}\int_{0}^{T}\|\mu\|_{H^{-s}}^{2}\ d\tau
\]
We use \eqref{payoffAssumption} and the fact that $v(T,x)=D_q G(x,m_T)\mu_T$ to obtain
\begin{equation}\label{eq:july04.2020.8} 
E_{v}(t)+ \frac{7}{8}\int_{t}^{T}\|\Lambda^{-s}\nabla v\|_{L^{2}}^{2}\ d\tau \leq {\kappa\over 2}\|\mu(T,\cdot)\|_{H^{-s-1}}^{2}
+cT \sup_{t\in[0,T]}\|v\|_{H^{-s}}^{2} + \frac{1}{12}\int_{0}^{T}\|\mu\|_{H^{-s}}^{2}\ d\tau
\end{equation}
Since each one of the two terms at the left hand side of  \eqref{eq:july04.2020.8} is nonnegative we conclude that they are less than or equal to the  expression at the right hand side of  \eqref{eq:july04.2020.8}. Hence, 
\begin{align} 
&\max\Big\{{1\over 2}\sup_{t\in[0,T]}\|v\|_{H^{-s}}^{2},\; \frac{7}{8}\int_0^{T}\|\Lambda^{-s}\nabla v\|_{L^{2}}^{2}\ d\tau\Big\}  \nonumber\\
\leq & {\kappa\over 2} \|\mu(T,\cdot)\|_{H^{-s-1}}^{2}+cT\sup_{t\in[0,T]}\|v\|_{H^{-s}}^{2} + \frac{1}{12}\int_{0}^{T}\|\mu\|_{H^{-s-1}}^{2}\ d\tau 
+ \frac{1}{12}\int_{0}^{T}\|\nabla \mu\|_{H^{-s-1}}^{2}\ d\tau. \label{eq:july04.2020.9}
\end{align}

Arguing as above, we combine \eqref{eq:july04.2020.4.13} \eqref{eq:july04.2020.4.15}  \eqref{eq:july04.2020.4.17} \eqref{eq:july04.2020.4.18} to obtain 
\begin{equation}\nonumber
E_{\mu}(t)+\frac{5}{6}\int_{0}^{T}\|\Lambda^{-s-1}\nabla\mu\|_{L^{2}}^{2}\ d\tau\leq E_{\mu}(0)+cT\sup_{t\in[0,T]}\|\mu\|_{H^{-s-1}}^{2}
+\frac{1}{8 \kappa}\int_{0}^{T}\|\Lambda^{-s}\nabla v\|_{L^{2}}^{2}\ d\tau.
\end{equation}
Taking into account the fact that each one of the expressions at the left and side of the previous inequality are non negative, we conclude that 
\begin{align}
&\max\bigg\{ {1\over 2} \sup_{t\in[0,T]}\|\mu\|_{H^{-s-1}}^{2},\;  \frac{5}{6}\int_0^T\|\Lambda^{-s-1}\nabla\mu\|^2_{L^2}d\tau\bigg\} \nonumber\\
&\leq  {1\over 2}\|\mu_{0}\|_{H^{-s-1}}^{2}+cT\sup_{t\in[0,T]}\|\mu\|_{H^{-s-1}}^{2}
+\frac{1}{8\kappa}\int_{0}^{T}\|\Lambda^{-s}\nabla v\|_{L^{2}}^{2}\ d\tau\label{eq:july04.2020.9ter}.
\end{align}
This gives an upper bound on $ \kappa \|\mu(T,\cdot)\|_{H^{-s-1}}^{2}$ which we can use in \eqref{eq:july04.2020.9} to obtain 
\begin{align} 
\max\Big\{{1\over 2}\sup_{t\in[0,T]}\|v\|_{H^{-s}}^{2},\; \frac{7}{8}\int_0^{T}\|\Lambda^{-s}\nabla v\|_{L^{2}}^{2}\ d\tau\Big\} 
\leq & {\kappa\over 2} \|\mu(0,\cdot)\|_{H^{-s-1}}^{2}+cT\sup_{t\in[0,T]}\|v\|_{H^{-s}}^{2} +\frac{1}{12}\int_{0}^{T}\|\mu\|_{H^{-s-1}}^{2}\ d\tau \nonumber\\
+ &  \frac{1}{12}\int_{0}^{T}\|\nabla \mu\|_{H^{-s-1}}^{2}\ d\tau+\frac{1}{8}\int_0^{T}\|\Lambda^{-s}\nabla v\|_{L^{2}}^{2}\ d\tau. \label{eq:july04.2020.9bis}
\end{align}
We combine \eqref{eq:july04.2020.9ter} and \eqref{eq:july04.2020.9bis} to conclude that 
\begin{align*} 
& {1\over 2} \sup_{t\in[0,T]}\|\mu\|_{H^{-s-1}}^{2}+ \frac{5}{6}\int_0^T\|\nabla\mu\|^2_{H^{-s-1}}d\tau+
{1\over 2}\sup_{t\in[0,T]}\|v\|_{H^{-s}}^{2}+ \frac{7}{8}\int_0^{T}\|\Lambda^{-s}\nabla v\|_{L^{2}}^{2}\ d\tau\\
\leq &   \|\mu_{0}\|_{H^{-s-1}}^{2}+cT\sup_{t\in[0,T]}\|\mu\|_{H^{-s-1}}^{2}
+\frac{1}{4\kappa}\int_{0}^{T}\|\Lambda^{-s}\nabla v\|_{L^{2}}^{2}\ d\tau\\
+ &\kappa \|\mu_0\|_{H^{-s-1}}^{2}+cT\sup_{t\in[0,T]}\|v\|_{H^{-s}}^{2} +\frac{1}{6}\int_{0}^{T}\|\mu\|_{H^{-s-1}}^{2}\ d\tau +   \frac{1}{6}\int_{0}^{T}\|\nabla \mu\|_{H^{-s-1}}^{2}\ d\tau+\frac{1}{4}\int_0^{T}\|\Lambda^{-s}\nabla v\|_{L^{2}}^{2}\ d\tau. 
\end{align*}
Thus, 
\begin{align*} 
& \bigg({1\over 3}-cT\bigg) \sup_{t\in[0,T]}\|\mu\|_{H^{-s-1}}^{2}+ \frac{2}{3}\int_0^T\|\nabla\mu\|^2_{H^{-s-1}}d\tau+
\bigg({1\over 2}-cT\bigg)\sup_{t\in[0,T]}\|v\|_{H^{-s}}^{2}+ \Big(\frac{5}{8}-{1\over 4\kappa}\Big)\int_0^{T}\|\nabla v\|_{H^{-s}}^{2}\ d\tau\\
\leq &  (\kappa+1) \|\mu_{0}\|_{H^{-s-1}}^{2}. 
\end{align*}
Since $\kappa\geq 1$ we complete the proof when $3cT<1.$  
\end{proof}

\begin{remark} (1) We remark that the constant $c$ in \eqref{estim:thm-H-s} depends on 
$$\|D_p\cH(t,\cdot,\nabla u,m)\|_{H^s}, \|\partial_q\cH(t,\cdot,\nabla u,m)\|_{H^s}, \|mD^2_{pp}\cH(t,\cdot,\nabla u,m)\|_{H^s}\ \ {\rm{and}}\ \ \|mD_p\partial_q\cH(t,\cdot,\nabla u,m)\|_{H^s}.$$

(2) The time horizon appearing in the previous theorem satisfies $T_{**}<\frac{1}{3c\kappa}.$ 

(3) in fact, we proved that  
 \[
\max\bigg\{\sup_{t\in[0,T]}\|v(t,\cdot)\|_{H^{-s}}; \sup_{t\in[0,T]}\|\mu\|_{H^{-s-1}}; \int_0^T\|\nabla\mu\|^2_{H^{-s-1}}d\tau; \int_0^{T}\|\nabla v\|_{H^{-s}}^{2}\ d\tau \bigg\}\leq c\|\mu_{0}\|_{H^{-s-1}}.
\]
which is a stronger conclusion than \eqref{estim:thm-H-s}.
\end{remark}

\begin{remark}
We remark that in the proof of Theorem \ref{estim:thm-H-s} we have only used that $\mu_0\in H^{-s-1}(\T^d)$ (i.e. we do not require for these estimates $\mu_0\in H^{s-1}(\T^d)$) and the regularity of solutions $(m,u)$ on the mean field game system \eqref{eq:mfg}. Therefore, we can conclude that in fact Theorem \ref{thm:v_H-s_bound} (combined with the techniques of the proof of Theorem \ref{step1Theorem}) provides an existence and uniqueness result for \eqref{eq:linearized} for initial data $\mu_0\in H^{-s-1}(\T^d).$
\end{remark}
\begin{corollary}\label{cor:weaker_v_mu}
Under the assumptions of Theorem \ref{step1Theorem} and $\mu_0\in H^{-s-1}(\T^d)$ (that replaces the assumption $\mu_0\in H^{s-1}(\T^d)$), there exists $T_{**}>0$ and a unique distributional solution $(v,\mu)$ to \eqref{eq:linearized}  such that if $T\in (0,T_{**})$ then 
$$
(v,\mu)\in L^\infty([0,T]; H^{-s}(\T^d))\times L^\infty([0,T]; H^{-s-1}(\T^d)).
$$
The proof of this result is straightforward, since the system \eqref{eq:linearized} is linear with smooth coefficients (and it follows from the estimates in the proof of Theorem \ref{estim:thm-H-s}), so we omit it.
\end{corollary}

{ The estimate in Theorem \ref{thm:v_H-s_bound} yields the following result.
\begin{lemma}\label{lem:Riesz}
The $v$ component of the solution to the linearized system \eqref{eq:linearized} (given by Theorem \ref{thm:linearized_intro}) can be represented as
\begin{equation}\label{eq:repr}
v(t,x)=\langle\mu_0,K(t,x,m_0,\cdot)\rangle_{H^{-s},H^{s}}, \ \forall (t,x)\in[0,T_*)\times\T^d,
\end{equation}
for a unique $K:[0,T_*)\times\T^d\times\sQ_R\times\T^d\to\R$. In particular $K(t,x,m_0,\cdot)\in H^s(\T^d)$.
\end{lemma}}
\begin{proof}
Let us denote by $T:H^{-s}(\T^d)\to H^{-s}(\T^d)$ the bounded linear operator such that $T(\mu_0)=v(t,\cdot)$ and let  $T':H^s(\T^d)\to H^s(\T^d)$ stand for its adjoint. 
Let $Q:H^{-s}(\T^d)\rightarrow H^{s}(\T^d)$ canonically, so that $Q^{-1}:H^{s}\rightarrow H^{-s}$ canonically.
(The operator $Q$ can be expressed simply in terms of the operator $\Lambda$ we have introduced previously;
since $f\in H^{-s}$ if and only if $\Lambda^{-s}f\in L^{2},$ and similarly for $H^{s},$ we may define $Qf=\Lambda^{-2s}f$
for any $f\in H^{-s}.$  Then $\Lambda^{s}Qf=\Lambda^{-s}f\in L^{2},$ and we see $Qf\in H^{s}.$  The inverse is similarly
given by $Q^{-1}=\Lambda^{2s}.$)

For any $\mu_{0}\in H^{-s}(\T^d),$ we have $QT\mu_{0}\in H^{s}.$  Since $s>\frac{d}{2},$ if we let $x$ be given,
then $\mu_{0}\mapsto (QT\mu_{0})(x)$ is a bounded linear functional on $H^{-s}.$  Thus this bounded linear functional
can be represented as the $L^{2}$ inner product with an element of $H^{s}.$  We write this as
\begin{equation}\nonumber
(QT\mu_{0})(x)=\langle\mu_{0}(y),K(x,y)\rangle_{(H^{-s},H^{s})_{y}},
\end{equation}
where we have established $K(x,y)\in H^{s}_{y}(\T^d)$ for any $x.$

If we fix $y_{0}$ and let $\mu_{0}$ be the Dirac mass supported at $y_{0},$ then we have
\begin{equation}\nonumber
(QT\mu_{0})(x)=K(x,y_{0}).
\end{equation}
We know that $QT\mu_{0}\in H^{s}_{x}(\T^d),$ so we have $K(x,y_{0})\in H^{s}_{x}(\T^d)$ for all $y_{0}.$

We let $\tilde{K}=Q^{-1}_{x}K,$ and we see that $\tilde{K}(\cdot,y)\in H^{-s}_{x}$ for all $y$ and
$\tilde{K}(x,\cdot)\in H^{s}_{y}$ for all $x.$
So far, we have the following:
\begin{align*}\nonumber
\langle T\mu_{0}(x),h(x)\rangle_{(H^{-s},H^{s})_{x}}
&=\langle Q^{-1}_{x}\langle \mu_{0}(y),K(x,y)\rangle_{(H^{-s},H^{s})_{y}},h(x)\rangle_{(H^{-s},H^{s})_{x}}
\\
&=\langle \langle \mu_{0}(y),Q^{-1}_{x}K(x,y)\rangle_{(H^{-s},H^{s})_{y}},h(x)\rangle_{(H^{-s},H^{s})_{x}}
\\
&=\langle\langle\mu_{0}(y),\tilde{K}(x,y)\rangle_{(H^{-s},H^{s})_{y}},h(x)\rangle_{(H^{-s},H^{s})_{x}}.
\end{align*}
We can then write this as
\begin{equation}\nonumber
\langle T\mu_{0}(x),h(x)\rangle_{(H^{-s},H^{s})_{x}}
=
\int\int\mu_{0}(y)\tilde{K}(x,y)h(x)\ dydx.
\end{equation}

Considering the 
subset of $H^{-s}(\T^d)\times H^{s}(\T^d)$ of pairs of functions $(\mu_0,h)$ each of which has exponentially decaying Fourier series,
we see that we may change the order of integration, finding
\begin{multline}\label{withIntegrals}
\langle T\mu_{0}(x),h(x)\rangle_{(H^{-s},H^{s})_{x}}
=
\int\int\mu_{0}(y)\tilde{K}(x,y)h(x)\ dxdy
\\
=\langle \mu_{0}(y),\langle \tilde{K}(x,y),h(x)\rangle_{(H^{-s},H^{s})_{x}}\rangle_{(H^{-s},H^{s})_{y}}
=\langle \mu_{0}(x),T'h(x)\rangle_{(H^{-s},H^{s})_{x}}.
\end{multline}
Since \eqref{withIntegrals} holds on a dense subset, this representation holds
for all of $H^{-s}(\T^d)\times H^{s}(\T^d).$ The thesis of the lemma follows by setting $K(t,x,m_0,\cdot):=\tilde K(x,\cdot)$.
\end{proof}

\section{The first variation of the master function}\label{sec:step3}


 In addition to the problem \eqref{eq:z} written for $z,$ we need to develop the 
corresponding problem for $\nu:=\tilde{m}-m-\mu.$  Note that the initial data for $\nu$ is a signed measure that has zero total mass, but is not equal to zero. An equation for $\nu$ is needed, since the PDE in \eqref{eq:z} involves the term $\partial_q\mathcal{H}\mu$, to which we need to add and subtract $\tilde{m}-m.$  One of the resulting terms will help to complete the Taylor expansion of $\mathcal{H},$ but
the other resulting term will include $\nu,$ which then must be estimated.

We take \eqref{eq:z} and perform the mentioned adding and subtracting:
\begin{multline}\nonumber
-\partial_{t}z-\Delta z+D_{p}\mathcal{H}(t,x,\nabla u,m)\cdot\nabla z=
-\mathcal{H}(t,x,\nabla\tilde{u},\tilde{m})+\mathcal{H}(t,x,\nabla u,m)
+D_{p}\mathcal{H}(t,x,\nabla u,m)\cdot\nabla(\tilde{u}-u)\\
-\partial_q\mathcal{H}(t,x,\nabla u,m)(\tilde{m}-m-\mu)+\partial_q\mathcal{H}(t,x,\nabla u,m)(\tilde{m}-m).
\end{multline}
We rearrange this to put the term with $\nu$ on the left-hand side:
\begin{multline}\label{zExplicit}
-\partial_{t}z-\Delta z+D_{p}\mathcal{H}(t,x,\nabla u,m)\cdot\nabla z+\partial_q\mathcal{H}(t,x,\nabla u,m)\nu=
-\mathcal{H}(t,x,\nabla\tilde{u},\tilde{m})+\mathcal{H}(t,x,\nabla u,m)\\
+D_{p}\mathcal{H}(t,x,\nabla u,m)\cdot\nabla(\tilde{u}-u)
+\partial_q\mathcal{H}(t,x,\nabla u,m)(\tilde{m}-m).
\end{multline}

We next need to form the $\nu$ equation.  Clearly we have $\partial_{t}\nu=\partial_{t}\tilde{m}-\partial_{t}m-\partial_{t}\mu,$
and we use \eqref{eq:mfg} for the $\tilde{m}$ and $m$ equations, and we use \eqref{eq:linearized} for the $\mu$ equation.
For brevity we will write $D_{p}\mathcal{H}=D_{p}\mathcal{H}(t,x,\nabla u,m)$ and 
$\widetilde{D_{p}\mathcal{H}}=D_{p}\mathcal{H}(t,x,\nabla\tilde{u},\tilde{m}),$ and similarly for $D_{pp}^{2}\mathcal{H}$
and $D_p\partial_q\mathcal{H}.$  To begin, just substituting for $\partial_{t}\tilde{m},$ $\partial_{t}m,$ and $\partial_{t}\mu,$
we have
\begin{equation}\nonumber
\partial_{t}\nu=\Delta\nu+\mathrm{div}(\tilde{m}\widetilde{D_{p}\mathcal{H}})
-\mathrm{div}(mD_{p}\mathcal{H})
-\mathrm{div}(\mu D_{p}\mathcal{H})-\mathrm{div}(mD_{pp}^{2}\mathcal{H}\nabla v)
+\mathrm{div}(mD_p\partial_q\mathcal{H}\mu).
\end{equation}
We next add and subtract a few times; for each of the two times on the right-hand side that $\mu$ appears, we add and
subtract $\tilde{m}-m$ to make $\nu$ appear, and in the one instance that $v$ appears on the right-hand side, we add and
subtract $\tilde{u}-u$ so as to bring out $z.$  These considerations yield the following:
\begin{multline}\label{nuExplicit}
\partial_{t}\nu=\Delta\nu-\mathrm{div}(\tilde{m}\widetilde{D_{p}\mathcal{H}})+\mathrm{div}(mD_{p}\mathcal{H})
+\mathrm{div}(\nu D_{p}\mathcal{H})-\mathrm{div}((\tilde{m}-m) D_{p}\mathcal{H})
\\
+\mathrm{div}(mD_{pp}^{2}\mathcal{H}\nabla z)-\mathrm{div}(mD_{pp}^{2}\mathcal{H}\nabla(\tilde{u}-u))
-\mathrm{div}(mD_p\partial_q\mathcal{H}\nu)+\mathrm{div}(mD_p\partial_q\mathcal{H}(\tilde{m}-m)).
\end{multline}

We restate the equations \eqref{zExplicit}, \eqref{nuExplicit} so that the terms which do not explicitly involve $z$ or $\nu$
are rephrased as forcing terms,
\begin{equation}\label{zWithF}
-\partial_{t}z-\Delta z+D_{p}\mathcal{H}(t,x,\nabla u,m)\cdot\nabla z+\partial_q\mathcal{H}(t,x,\nabla u,m)\nu+
F_{1}=0,
\end{equation}
\begin{equation}\label{nuWithF}
\partial_{t}\nu=\Delta\nu
+\mathrm{div}(\nu D_{p}\mathcal{H})
+\mathrm{div}(mD_{pp}^{2}\mathcal{H}\nabla z)
-\mathrm{div}(mD_p\partial_q\mathcal{H}\nu)+F_{2}.
\end{equation}
Here, we recall the formulas for $F_{1}$ and $F_{2}$ from \eqref{f1Bound} and \eqref{f2Bound}.

\begin{theorem}\label{thm:z_regularity}  
Let $s$ satisfy \eqref{hyp:s} and let  
$r$ be given as in \eqref{con:r}. { Let $R>0$.} Let $(u,m), (\tilde u, \tilde m)$ be solutions to \eqref{eq:mfg}, with initial data $m_0,\tilde m_0\in { \sQ_R}$, respectively, given by Theorem \ref{thm:david_old}. Let $(v,\nu)$ be the solution to \eqref{eq:linearized} with initial data $\mu_0:=\tilde m_0-m_0$, given by Theorem \ref{thm:linearized_intro}.  There exists $c>0$ and $T_{***}>0$ such that if $T\in(0,T_{***}),$ then
\begin{equation}\label{estim:z_regularity}
\sup_{t\in[0,T]}\|z\|_{H^{r}}^{2}\leq c\|m_{0}-\tilde{m}_{0}\|_{H^{-1}}^{5/2}.
\end{equation}
\end{theorem}

\begin{proof}
We will make energy estimates for $z\in H^{r}(\T^d)$ and $\nu\in H^{r-1}(\T^d),$ for the given $r.$  We therefore define
\begin{equation}\nonumber
E_{z}=\frac{1}{2}\int_{\mathbb{T}^{d}}(\Lambda^{r} z)^{2}\ dx,
\end{equation}
\begin{equation}\nonumber
E_{\nu}=\frac{1}{2}\int_{\mathbb{T}^{d}}(\Lambda^{r-1}\nu)^{2}\ dx.
\end{equation}
{Let us notice that by the assumptions on $s$ in \eqref{hyp:s} and the choice of $r$ in \eqref{con:r}, $H^r(\T^d)$ is an algebra}.

Then we have
\begin{equation}\nonumber
\frac{dE_{z}}{dt}=\int_{\mathbb{T}^{d}}(\Lambda^{r}z)(\Lambda^{r}\partial_{t}z)\ dx.
\end{equation}
We substitute from \eqref{zWithF}, to get
\begin{multline}\nonumber
\frac{dE_{z}}{dt}=-\int_{\mathbb{T}^{d}}(\Lambda^{r}z)(\Lambda^{r}\Delta z)\ dx
+\int_{\mathbb{T}^{d}}(\Lambda^{r}z)(\Lambda^{r}(D_{p}\mathcal{H}\cdot\nabla z))\ dx
\\
+\int_{\mathbb{T}^{d}}(\Lambda^{r}z)(\Lambda^{r}(\partial_q\mathcal{H}\nu))\ dx
+\int_{\mathbb{T}^{d}}(\Lambda^{r}z)(\Lambda^{r}F_{1})\ dx
= I_{z}+II_{z}+III_{z}+IV_{z}.
\end{multline}
We integrate in time from time $t$ to time $T$ to find
\begin{equation}\label{EzToSubstitute}
E_{z}(t)=E_{z}(T)-\int_{t}^{T}I_{z}\ d\tau-\int_{t}^{T}II_{z}\ d\tau-\int_{t}^{T}III_{z}\ d\tau-\int_{t}^{T}IV_{z}\ d\tau.
\end{equation}

For the term involving $I_{z},$ we integrate by parts with respect to the spatial variable,
\begin{equation}\label{final1z}
-\int_{t}^{T}I_{z}\ d\tau=-\int_{t}^{T}\int_{\mathbb{T}^{d}}(\Lambda^{r}\nabla z)^{2}\ dxd\tau.
\end{equation}
For the term involving $II_{z},$ we use Young's inequality with parameter $\sigma_{1},$
\begin{equation}\nonumber
-\int_{t}^{T}II_{z}\ d\tau \leq \int_{t}^{T}\frac{1}{2\sigma_{1}}\|z\|_{H^{r}}^{2}\ d\tau
+\int_{t}^{T}\frac{\sigma_{1}}{2}\|(D_{p}\mathcal{H})\nabla z\|_{H^{r}}^{2}\ d\tau.
\end{equation}
We continue by bounding $D_{p}\mathcal{H}$ by a constant, since this depends only on the solution
$(u,m)$ of the original problem.  We also bound the integrals on the right-hand side, using a supremum in time 
for the first of these, and taking a larger domain of integration for the second of these:
\begin{equation}\label{final2z}
-\int_{t}^{T}II_{z}\ d\tau\leq \frac{T}{2\sigma_{1}}\sup_{t\in[0,T]}\|z\|_{H^{r}}^{2}
+\frac{c\sigma_{1}}{2}\int_{0}^{T}\int_{\mathbb{T}^{d}}(\Lambda^{r}\nabla z)^{2}\ dxd\tau.
\end{equation}
We treat $III_{z}$ similarly, applying Young's inequality with positive parameter $\sigma_{2},$
\begin{equation}\nonumber
-\int_{t}^{T}III_{z}\ d\tau
\leq
\int_{t}^{T}\frac{1}{2\sigma_{2}}\|z\|_{H^{r}}^{2}\ d\tau
+\int_{t}^{T}\frac{\sigma_{2}}{2}\|(\partial_q\mathcal{H})\nu\|_{H^{r}}^{2}\ d\tau.
\end{equation}
Again, continuing as we have for the previous term, we find
\begin{equation}\nonumber
-\int_{t}^{T}III_{z}\ d\tau\leq\frac{T}{2\sigma_{2}}\sup_{t\in[0,T]}\|z\|_{H^{r}}^{2}
+\frac{c\sigma_{2}}{2}\int_{0}^{T}\int_{\mathbb{T}^{d}}(\Lambda^{r}\nu)^{2}\ dxd\tau.
\end{equation}
We next wish to change the form of the final integral on the right-hand side, to have $(\Lambda^{r-1}\nabla\nu)^{2}$
rather than $(\Lambda^{r}\nu)^{2}$. Using the inequality
$$\int_{\T^d}(\Lambda^r\nu)^2 dx\le c\left(\|\nu\|_{H^{r-1}} +\int_{\T^d}(\Lambda^{r-1}\nabla\nu)^2 dx\right)$$
we obtain the bound
\begin{equation}\label{final3z}
-\int_{t}^{T}III_{z}\ d\tau\leq\frac{T}{2\sigma_{2}}\sup_{t\in[0,T]}\|z\|_{H^{r}}^{2}
+c\sigma_{2}T\sup_{t\in[0,T]}\|\nu\|_{H^{r-1}}^{2}
+c\sigma_{2}\int_{0}^{T}\int_{\mathbb{T}^{d}}(\Lambda^{r-1}\nabla\nu)^{2}\ dxd\tau.
\end{equation}
For the term involving $IV_{z},$ we use Young's inequality without parameter,
\begin{equation}\nonumber
-\int_{t}^{T}IV_{z}\ d\tau
\leq \frac{T}{2}\sup_{t\in[0,T]}\|z\|_{H^{r}}^{2}
+\frac{T}{2}\sup_{t\in[0,T]}\|F_{1}\|_{H^{r}}^{2}.
\end{equation}
We then use \eqref{f1Bound}, finding
\begin{equation}\label{final4z}
-\int_{t}^{T}IV_{z}\ d\tau
\leq \frac{T}{2}\sup_{t\in[0,T]}\|z\|_{H^{r}}^{2}
+c\sup_{t\in[0,T]}(\|m-\tilde{m}\|_{H^{r}}^{2}+\|u-\tilde{u}\|_{H^{r+1}}^{2}).
\end{equation}

We combine \eqref{EzToSubstitute}, using the definition of $E_{z},$ with \eqref{final1z}, \eqref{final2z},
\eqref{final3z}, and \eqref{final4z}.  This yields
\begin{multline}\nonumber
\frac{1}{2}\|z(t,\cdot)\|_{H^{r}}^{2}+\int_{t}^{T}\int_{\mathbb{T}^{d}}(\Lambda^{r}\nabla z)^{2}\ dxd\tau
\\
\leq\frac{1}{2}\|z(T,\cdot)\|_{H^{r}}^{2}
+\frac{T}{2}\left(\frac{1}{\sigma_{1}}+\frac{1}{\sigma_{2}}+1\right)\sup_{t\in[0,T]}\|z\|_{H^{r}}^{2}
+c\sigma_{2}T\sup_{t\in[0,T]}\|\nu\|_{H^{r-1}}^{2}
\\
+c\sigma_{1}\int_{0}^{T}\int_{\mathbb{T}^{d}}(\Lambda^{r}\nabla z)^{2}\ dxd\tau
+c\sigma_{2}\int_{0}^{T}\int_{\mathbb{T}^{d}}(\Lambda^{r-1}\nabla\nu)^{2}\ dxd\tau
\\
+c\sup_{t\in[0,T]}\left(\|m-\tilde{m}\|_{H^{r}}^{2}+\|u-\tilde{u}\|_{H^{r+1}}^{2}\right).
\end{multline}
Since by \eqref{hyp:G-C11} the data $z(T,\cdot)$ is also quadratic in $m-\tilde{m},$ this simplifies slightly as
\begin{multline}\nonumber
\frac{1}{2}\|z(t,\cdot)\|_{H^{r}}^{2}+\int_{t}^{T}\int_{\mathbb{T}^{d}}(\Lambda^{r}\nabla z)^{2}\ dxd\tau
\\
\leq\frac{T}{2}\left(\frac{1}{\sigma_{1}}+\frac{1}{\sigma_{2}}+1\right)\sup_{t\in[0,T]}\|z\|_{H^{r}}^{2}
+c\sigma_{2}T\sup_{t\in[0,T]}\|\nu\|_{H^{r-1}}^{2}
\\
+c\sigma_{1}\int_{0}^{T}\int_{\mathbb{T}^{d}}(\Lambda^{r}\nabla z)^{2}\ dxd\tau
+c\sigma_{2}\int_{0}^{T}\int_{\mathbb{T}^{d}}(\Lambda^{r-1}\nabla\nu)^{2}\ dxd\tau
\\
+c\sup_{t\in[0,T]}\left(\|m-\tilde{m}\|_{H^{r}}^{2}+\|u-\tilde{u}\|_{H^{r+1}}^{2}\right).
\end{multline}
We then treat the terms on the left-hand side separately, and take the supremum with respect to time, finding
\begin{multline}\label{firstSupToAdd}
\frac{1}{2}\sup_{t\in[0,T]}\|z\|_{H^{r}}^{2}
\leq\frac{T}{2}\left(\frac{1}{\sigma_{1}}+\frac{1}{\sigma_{2}}+1\right)\sup_{t\in[0,T]}\|z\|_{H^{r}}^{2}
+c\sigma_{2}T\sup_{t\in[0,T]}\|\nu\|_{H^{r-1}}^{2}
\\
+c\sigma_{1}\int_{0}^{T}\int_{\mathbb{T}^{d}}(\Lambda^{r}\nabla z)^{2}\ dxd\tau
+c\sigma_{2}\int_{0}^{T}\int_{\mathbb{T}^{d}}(\Lambda^{r-1}\nabla\nu)^{2}\ dxd\tau
\\
+c\sup_{t\in[0,T]}\left(\|m-\tilde{m}\|_{H^{r}}^{2}+\|u-\tilde{u}\|_{H^{r+1}}^{2}\right).
\end{multline}
as well as
\begin{multline}\label{firstTimeIntegralToAdd}
\int_{0}^{T}\int_{\mathbb{T}^{d}}(\Lambda^{r}\nabla z)^{2}\ dxd\tau
\leq\frac{T}{2}\left(\frac{1}{\sigma_{1}}+\frac{1}{\sigma_{2}}+1\right)\sup_{t\in[0,T]}\|z\|_{H^{r}}^{2}
+c\sigma_{2}T\sup_{t\in[0,T]}\|\nu\|_{H^{r-1}}^{2}
\\
+c\sigma_{1}\int_{0}^{T}\int_{\mathbb{T}^{d}}(\Lambda^{r}\nabla z)^{2}\ dxd\tau
+c\sigma_{2}\int_{0}^{T}\int_{\mathbb{T}^{d}}(\Lambda^{r-1}\nabla\nu)^{2}\ dxd\tau
\\
+c\sup_{t\in[0,T]}\left(\|m-\tilde{m}\|_{H^{r}}^{2}+\|u-\tilde{u}\|_{H^{r+1}}^{2}\right).
\end{multline}

We next write out the time derivative of $E_{\nu},$
\begin{equation}\nonumber
\frac{dE_{\nu}}{dt}=\int_{\mathbb{T}^{d}}(\Lambda^{r-1}\nu)(\Lambda^{r-1}\partial_{t}\nu)\ dx.
\end{equation}
We substitute from \eqref{nuWithF} to find
\begin{multline}\nonumber
\frac{dE_{\nu}}{dt}=\int_{\mathbb{T}^{d}}(\Lambda^{r-1}\nu)(\Lambda^{r-1}\Delta\nu)\ dx
+\int_{\mathbb{T}^{d}}(\Lambda^{r-1}\nu)(\Lambda^{r-1}\mathrm{div}(\nu D_{p}\mathcal{H}))\ dx
\\
+\int_{\mathbb{T}^{d}}(\Lambda^{r-1}\nu)(\Lambda^{r-1}\mathrm{div}(mD_{pp}^{2}\mathcal{H}\nabla z))\ dx
-\int_{\mathbb{T}^{d}}(\Lambda^{r-1}\nu)(\Lambda^{r-1}\mathrm{div}(mD_p\partial_q\mathcal{H} \nu))\ dx
\\
+\int_{\mathbb{T}^{d}}(\Lambda^{r-1}\nu)(\Lambda^{r-1}F_{2})\ dx=I_{\nu}+II_{\nu}+III_{\nu}+IV_{\nu}+V_{\nu}.
\end{multline}
We integrate in time from time zero until time $T,$ and we use $E_{\nu}(0)=0,$ finding
\begin{equation}\label{ENuToSubstitute}
E_{\nu}(t)=\int_{0}^{t}I_{\nu}\ d\tau+\int_{0}^{t}II_{\nu}\ d\tau
+\int_{0}^{t}III_{\nu}\ d\tau+\int_{0}^{t}IV_{\nu}\ d\tau+\int_{0}^{t}V_{\nu}\ d\tau.
\end{equation}

For the term involving $I_{\nu},$ we integrate by parts with respect to the spatial variable,
\begin{equation}\label{final1Nu}
\int_{0}^{t}I_{\nu}\ d\tau=-\int_{0}^{t}\int_{\mathbb{T}^{d}}(\Lambda^{r-1}\nabla\nu)^{2}\ dxd\tau.
\end{equation}
For the term involving $II_{\nu},$ we use Young's inequality with paramter $\sigma_{3},$
and we proceed as we did above for the terms involving $II_{z}$ and $III_{z},$
\begin{equation}\label{final2Nu}
\int_{0}^{t}|II_{\nu}|\ d\tau \leq \frac{T}{2\sigma_{3}}\sup_{t\in[0,T]}\|\nu\|_{H^{r-1}}^{2}
+\frac{c\sigma_{3}}{2}\int_{0}^{T}\int_{\mathbb{T}^{d}}(\Lambda^{r-1}\nabla\nu)^{2}\ dxd\tau.
\end{equation}
For the term involving $III_{\nu},$ we use Young's inequality with parameter $\sigma_{4},$
\begin{equation}\label{final3Nu}
\int_{0}^{t}| III_{\nu}|\ d\tau\leq\frac{T}{2\sigma_{4}}\sup_{t\in[0,T]}\|\nu\|_{H^{r-1}}^{2}
+\frac{c\sigma_{4}}{2}\int_{0}^{T}\int_{\mathbb{T}^{d}}(\Lambda^{r}\nabla z)^{2}\ dxd\tau.
\end{equation}
For the term involving $IV_{\nu}$ we continue in the same manner, using Young's inequality with positive parameter
$\sigma_{5},$
\begin{equation}\label{final4Nu}
\int_{0}^{t}|IV_{\nu}|\ d\tau\leq\frac{T}{2\sigma_{5}}\sup_{t\in[0,T]}\|\nu\|_{H^{r-1}}^{2}
+\frac{c\sigma_{5}}{2}\int_{0}^{T}\int_{\mathbb{T}^{d}}(\Lambda^{r-1}\nabla\nu)^{2}\ dxd\tau.
\end{equation}
Finally, we estimate the term involving $V_{\nu}$ in the same manner we estimated the above term involving $IV_{z},$
\begin{equation}\nonumber
\int_{0}^{t}|V_{\nu}|\ d\tau\leq \frac{T}{2}\sup_{t\in[0,T]}\|\nu\|_{H^{r-1}}^{2}
+\frac{T}{2}\sup_{t\in[0,T]}\|F_{2}\|_{H^{r-1}}^{2}.
\end{equation}
We use \eqref{f2Bound} with this, yielding
\begin{equation}\label{final5Nu}
\int_{0}^{t}|V_{\nu}|\ d\tau\leq \frac{T}{2}\sup_{t\in[0,T]}\|\nu\|_{H^{r-1}}^{2}
+c\sup_{t\in[0,T]}\left(\|m-\tilde{m}\|_{H^{r-1}}^{4}+\|u-\tilde{u}\|_{H^{r}}^{4}\right).
\end{equation}

We use the definition of $E_{\nu}$ in \eqref{ENuToSubstitute}, and we substitute from
\eqref{final1Nu}, \eqref{final2Nu}, \eqref{final3Nu}, \eqref{final4Nu}, and \eqref{final5Nu}.  This yields
\begin{multline}\nonumber
\frac{1}{2}\|\nu(t,\cdot)\|_{H^{r-1}}^{2}+\int_{0}^{t}\int_{\mathbb{T}^{d}}(\Lambda^{r-1}\nabla\nu)^{2}\ dxd\tau
\\
\leq
\frac{T}{2}\left(\frac{1}{\sigma_{3}}+\frac{1}{\sigma_{4}}+\frac{1}{\sigma_{5}}+1\right)
\sup_{t\in[0,T]}\|\nu\|_{H^{r-1}}^{2}
\\
+c(\sigma_{3}+\sigma_{5})\int_{0}^{T}\int_{\mathbb{T}^{d}}(\Lambda^{r-1}\nabla\nu)^{2}\ dxd\tau
+c\sigma_{4}\int_{0}^{T}\int_{\mathbb{T}^{d}}(\Lambda^{r}\nabla z)^{2}\ dxd\tau
\\
+c\sup_{t\in[0,T]}\left(\|m-\tilde{m}\|_{H^{r-1}}^{4}+\|u-\tilde{u}\|_{H^{r}}^{4}\right).
\end{multline}
Treating the terms on the left-hand side separately, and taking the supremum with respect to time, we
thus find
\begin{multline}\label{secondSupToAdd}
\frac{1}{2}\sup_{t\in[0,T]}\|\nu\|_{H^{r-1}}^{2}
\leq
\frac{T}{2}\left(\frac{1}{\sigma_{3}}+\frac{1}{\sigma_{4}}+\frac{1}{\sigma_{5}}+1\right)
\sup_{t\in[0,T]}\|\nu\|_{H^{r-1}}^{2}
\\
+c(\sigma_{3}+\sigma_{5})\int_{0}^{T}\int_{\mathbb{T}^{d}}(\Lambda^{r-1}\nabla\nu)^{2}\ dxd\tau
+c\sigma_{4}\int_{0}^{T}\int_{\mathbb{T}^{d}}(\Lambda^{r}\nabla z)^{2}\ dxd\tau
\\
+c\sup_{t\in[0,T]}\left(\|m-\tilde{m}\|_{H^{r-1}}^{4}+\|u-\tilde{u}\|_{H^{r}}^{4}\right),
\end{multline}
and also
\begin{multline}\label{secondTimeIntegralToAdd}
\int_{0}^{T}\int_{\mathbb{T}}^{d}(\Lambda^{r-1}\nabla\nu)^{2}\ dxd\tau
\leq
\frac{T}{2}\left(\frac{1}{\sigma_{3}}+\frac{1}{\sigma_{4}}+\frac{1}{\sigma_{5}}+1\right)
\sup_{t\in[0,T]}\|\nu\|_{H^{r-1}}^{2}
\\
+c(\sigma_{3}+\sigma_{5})\int_{0}^{T}\int_{\mathbb{T}^{d}}(\Lambda^{r-1}\nabla\nu)^{2}\ dxd\tau
+c\sigma_{4}\int_{0}^{T}\int_{\mathbb{T}^{d}}(\Lambda^{r}\nabla z)^{2}\ dxd\tau
\\
+c\sup_{t\in[0,T]}\left(\|m-\tilde{m}\|_{H^{r-1}}^{4}+\|u-\tilde{u}\|_{H^{r}}^{4}\right).
\end{multline}

We add \eqref{firstTimeIntegralToAdd} and \eqref{secondTimeIntegralToAdd}, finding
\begin{multline}\nonumber
\int_{0}^{T}\int_{\mathbb{T}^{d}}(\Lambda^{r}\nabla z)^{2}\ dxd\tau
+\int_{0}^{T}\int_{\mathbb{T}^{d}}(\Lambda^{r-1}\nabla\nu)^{2}\ dxd\tau
\\
\leq
\frac{T}{2}\left(\frac{1}{\sigma_{1}}+\frac{1}{\sigma_{2}}+1\right)\sup_{t\in[0,T]}\|z\|_{H^{r}}^{2}
+\left(\frac{T}{2\sigma_{3}}+\frac{T}{2\sigma_{4}}+\frac{T}{2\sigma_{5}}+\frac{T}{2}+c\sigma_{2}T\right)
\sup_{t\in[0,T]}\|\nu\|_{H^{r-1}}^{2}
\\
+c(\sigma_{1}+\sigma_{4})\int_{0}^{T}\int_{\mathbb{T}^{d}}(\Lambda^{r}\nabla z)^{2}\ dxd\tau
+c(\sigma_{2}+\sigma_{3}+\sigma_{5})\int_{0}^{T}\int_{\mathbb{T}^{d}}(\Lambda^{r-1}\nabla\nu)^{2}\ dxd\tau
\\
+c\sup_{t\in[0,T]}\left(\|m-\tilde{m}\|_{H^{r}}^{4}+\|u-\tilde{u}\|_{H^{r+1}}^{4}\right).
\end{multline}
We choose the constants so that
\begin{equation}\nonumber
c(\sigma_{1}+\sigma_{4})=\frac{1}{2},\qquad
c(\sigma_{2}+\sigma_{3}+\sigma_{5})=\frac{1}{2}.
\end{equation}
We then conclude
\begin{multline}\label{integralConclusion}
\int_{0}^{T}\int_{\mathbb{T}^{d}}(\Lambda^{r}\nabla z)^{2}\ dxd\tau
+\int_{0}^{T}\int_{\mathbb{T}^{d}}(\Lambda^{r-1}\nabla\nu)^{2}\ dxd\tau
\\
\leq cT\sup_{t\in[0,T]}\|z\|_{H^{r}}^{2}+cT\sup_{t\in[0,T]}\|\nu\|_{H^{r-1}}^{2}
+c\sup_{t\in[0,T]}\left(\|m-\tilde{m}\|_{H^{r}}^{4}+\|u-\tilde{u}\|_{H^{r+1}}^{4}\right).
\end{multline}

We next use the bound \eqref{integralConclusion} on the right-hand sides of \eqref{firstSupToAdd} and 
\eqref{secondSupToAdd}, finding
\begin{multline}\nonumber
\sup_{t\in[0,T]}\|z\|_{H^{r}}^{2}+\sup_{t\in[0,T]}\|\nu\|_{H^{r-1}}^{2}
\leq cT\sup_{t\in[0,T]}\|z\|_{H^{r}}^{2}+cT\sup_{t\in[0,T]}\|\nu\|_{H^{r-1}}^{2}
\\
+c\sup_{t\in[0,T]}\left(\|m-\tilde{m}\|_{H^{r}}^{4}+\|u-\tilde{u}\|_{H^{r+1}}^{4}\right).
\end{multline}
Taking $T$ sufficiently small, we have
\begin{equation}\nonumber
\sup_{t\in[0,T]}\|z\|_{H^{r}}^{2}+\sup_{t\in[0,T]}\|\nu\|_{H^{r-1}}^{2}
\leq c\sup_{t\in[0,T]}\left(\|m-\tilde{m}\|_{H^{r}}^{4}+\|u-\tilde{u}\|_{H^{r+1}}^{4}\right).
\end{equation}
Neglecting the $\nu$-term on the left-hand side, this establishes
\begin{equation}\label{zGoal}
\sup_{t\in[0,T]}\|z\|_{H^{r}}^{2}\leq c\sup_{t\in[0,T]}\left(\|m-\tilde{m}\|_{H^{r}}^{2}+\|u-\tilde{u}\|_{H^{r+1}}^{2}\right).
\end{equation}

We will use the Sobolev interpolation inequality \eqref{sobolevInterpolation}, and we will do this in two steps.
First we take $\alpha=r$ and $\beta=4r.$  Then we have
\begin{equation}\label{firstUseInterp}
\|f\|_{H^{r}}^{4}\leq c \|f\|_{L^{2}}^{3}\|f\|_{H^{4r}}.
\end{equation}
We do this because, Theorem \ref{uniquenessTheorem} yields
\begin{equation}\label{priorStability}
\sup_{t\in[0,T]}\left(\|m-\tilde{m}\|_{L^{2}}^{4}+\|u-\tilde{u}\|_{H^{1}}^{4}\right)\leq c\|m_{0}-\tilde{m}_{0}\|_{L^{2}}^{4}.
\end{equation}
Although this estimate is the consequence of results from \cite{Amb18:2},  for completeness, we provide its proof in Appendix \ref{sec:appendix_estimate}. So, combining \eqref{zGoal} with \eqref{firstUseInterp} and \eqref{priorStability}, we have
\begin{equation}\nonumber
\sup_{t\in[0,T]}\|z\|_{H^{r}}^{2}\leq c\left(\sup_{t\in[0,T]}\|m-\tilde{m}\|_{L^{2}}^{3}+\|u-\tilde{u}\|_{H^{1}}^{3}\right)
\leq c\|m_{0}-\tilde{m}_{0}\|_{L^{2}}^{3},
\end{equation}
where we have used that $m$ and $\tilde{m}$ remain bounded in $H^{4r},$ 
and $u$ and $\tilde{u}$ remain bounded in $H^{4r+1},$
throughout the time interval $[0,T].$
We then make another use of \eqref{sobolevInterpolation}, with $\alpha=1$ and $\beta=6,$
\begin{equation}\nonumber
\|f\|_{L^{2}}^{3}=\|\Lambda^{-1}f\|_{H^{1}}^{3}\leq\|\Lambda^{-1}f\|_{L^{2}}^{5/2}\|\Lambda^{-1}f\|_{H^{6}}^{1/2}
=c\|f\|_{H^{-1}}^{5/2}\|f\|_{H^{5}}^{1/2}.
\end{equation}
With $m_{0}$ and $\tilde{m}_{0}$ bounded in $H^{5},$ this implies
\begin{equation}\nonumber
\sup_{t\in[0,T]}\|z\|_{H^{r}}^{2}\leq c\|m_{0}-\tilde{m}_{0}\|_{H^{-1}}^{5/2}.
\end{equation}
As desired, we have established $z=o(\|m_{0}-\tilde{m}_{0}\|_{H^{-1}}).$  This completes the proof.
\end{proof}


\begin{remark} (1) We notice that the constant $c$ in the estimate \eqref{estim:z_regularity} is depending on $\|D_p\cH(t,\cdot,\nabla u,m)\|_{H^r}$, $\|\partial_q\cH(t,\cdot,\nabla u,m)\|_{H^r},$ $\|mD^2_{pp}\cH(t,\cdot,\nabla u,m)\|_{H^r}$, $\|mD_p\partial_q\cH(t,\cdot,\nabla u,m)\|_{H^r}$, $\|m\|_{H^{4r}}$, $\|\tilde m\|_{H^{4r}}$, $\|u\|_{H^{4r+1}}$, $\|\tilde u\|_{H^{4r+1}}$ and on the constants from the assumptions \eqref{f1Bound}-\eqref{f2Bound}. These last constants depend in particular on 
$$\|D^2\cH(t,\cdot,m,\nabla u)\|_{H^r}\ \ {\rm{and}}\ \ \|mD^3\cH(t,\cdot,m,\nabla u)\|_{H^r}.$$

(2) The time horizon appearing in Theorem \ref{thm:z_regularity} satisfies $T_{***}<\frac{1}{c}.$
\end{remark}

\subsection{Regularity of the master function}\label{subsec:regularity}

We show now how the results from Sections \ref{sec:step1}, \ref{sec:step2} and \ref{sec:step3} will yield the necessary regularity of the master functions $U$. 

Let us recall our standing assumption \eqref{hyp:s}. We recall $r$ from \eqref{con:r} which in particular satisfies $r>\lceil \frac{d}{2}\rceil$.

\begin{theorem}\label{thm:U_reg}
Let { $R>0$ and let} $U:[0,T]\times\T^d\times{ \sQ_R}\to\R$ be defined in \eqref{def:U}. Then 
\begin{itemize}
\item[(i)] $U(\cdot,\cdot,m_0)\in L^\infty([0,T]; H^s(\T^d))\cap C([0,T]; C^{2,\a}(\T^d))\cap C^{1}([0,T]\times\T^d)$ for all $m_0\in { \sQ_R}$ and $t\in[0,T]$. 
\item[(ii)] For all $(t,x)\in[0,T]\times\T^d$, $y\in\T^d$ and $m_0\in { \sQ_R}$, the function $\frac{\delta U}{\delta m} (t,x,m_0)(y)$ is well defined. Moreover,  $\frac{\delta U}{\delta m} (t,x,m_0)(\cdot)$ defines a continuous linear operator on $H^{-s}(\T^d).$
\item[(iii)] The map $m\mapsto U(t,x,m)$ is G\^ateaux differentiable in $H^{s}(\T^d)$, uniformly with respect to $t,x\in[0,T]\times\T^d$. In particular,  there exists $C>0$ such that 
$$
\sup_{t\in[0,T],x\in\T^d}\left| U(t,x,\tilde m_0) -U(t,x,m_0) - \int_{\T^d}\frac{\delta U}{\delta m}(t,x,m_0)(y) \dd(\tilde m_0-m_0)(y) \right|  \le C\|m_0-\tilde m_0\|_{H^{s}}^{\frac54},
$$
for all $\tilde m_0, m_0\in { \sQ_R}$.
\item[(iv)] $H^s(\T^d)\ni m\mapsto \frac{\d U}{\d m}(t,\cdot,m)(y)\in H^{-s}(\T^d)$ is continuous uniformly with respect to $(t,y)\in[0,T]\times\T^d$.
\item[(v)] The functions $\T^d\ni y\mapsto \frac{\d U}{\d m}(t,\cdot,m_0)(y)\in H^{-s}(\T^d)$, $\T^d\ni y\mapsto \nabla_y\frac{\d U}{\d m}(t,\cdot,m_0)(y)\in H^{-s}(\T^d)$ and $\T^d\ni y\mapsto D^2_{yy}\frac{\d U}{\d m}(t,\cdot,m_0)(y)\in H^{-s}(\T^d)$ are Lipschitz continuous, for any $(t,m_0)\in[0,T]\times { \sQ_R}.$
\end{itemize}
\end{theorem}

\begin{proof}
(i) By the definition of $U$ in \eqref{def:U}, for each fixed $m_0\in { \sQ_R}$, $U(\cdot,\cdot,m_0)$ inherits the regularity properties of $u$. The result follows from Corollary \ref{cor:diff_u_m}.

\smallskip

(ii-iii) { Let us recall $K$ from the Riesz type representation, provided in Lemma \ref{lem:Riesz}.}

By \eqref{def:z}, \eqref{def:U} and the estimates in Theorem \ref{thm:z_regularity}, for $m_0,\tilde m_0\in { \sQ_R},$ we see that there exists a constant $C>0$ such that 
$$\sup_{t\in[0,T]}\Big\| U(t,\cdot,\tilde m_0) -U(t,\cdot,m_0) - \int_{\T^d}K(t,\cdot,m_0)(y)  \dd(\tilde m_0-m_0)(y) \Big\|_{H^r}  \le C \|m_0-\tilde m_0\|_{H^{-1}}^{\frac54}.$$
If $r>\lceil \frac{d}{2}\rceil$, we have in addition
$$\sup_{t\in[0,T]}\sup_{x\in\T^d}\Big| U(t,x,\tilde m_0) -U(t,x,m_0) - \int_{\T^d}K(t,x,m_0)(y)  \dd(\tilde m_0-m_0)(y) \Big|  \le C \|m_0-\tilde m_0\|_{H^{-1}}^{\frac54}\le C \|m_0-\tilde m_0\|_{H^{s}}^{\frac54}.$$
By the regularity of $K$, this means that $U(t,x,\cdot)$ is G\^ateaux differentiable in $H^{-s}(\T^d)\cap H^s(\T^d)= H^s(\T^d)$, uniformly with respect to $(t,x)\in[0,T]\times\T^d$. 

Now, let us fix $y_0\in\T^d$. Since $\sP(\T^d) \hookrightarrow H^{-s-1}(\T^d)$, we have that $\mu_0=\d_{y_0}\in H^{-s-1}(\T^d).$ And so, by the previous representation one obtains 
$$
v(t,\cdot) = \int_{\T^d}K(t,x,m_0)(y)\dd \mu_0(y)=K(t,x,m_0)(y_0) 
$$
and so, we can set 
$$
\frac{\d U}{\d m}(t,x,m_0)(y_0):=K(t,x,m_0)(y_0) ,
$$
and this expression is meaningful pointwise.

\smallskip

(iv) 

\medskip

By Corollary \ref{cor:weaker_v_mu} for $\mu_0\in H^{-s-1}(\T^d)$ we have that $v\in L^\infty([0,T]; H^{-s}(\T^d))$. Therefore, \eqref{eq:repr} yields that
$$
H^{-s}(\T^d)\ni v(t,\cdot) = \int_{\T^d}\frac{\d U}{\d m}(t,\cdot,m_0)(y)\dd \mu_0(y),
$$ 
for all $t\in[0,T].$ Since the solution $(u,m)$ of \eqref{eq:mfg} depends continuously on $m_0\in { \sQ_R}$, and so the coefficient functions in \eqref{eq:linearized} depend also continuously on $m_0$, we have that for each fixed $\mu_0\in H^{-s-1}(\T^d)$, $v(t,\cdot)$ depends continuously on $m_0\in { \sQ_R}$, uniformly in $t\in[0,T]$. Now, let us fix $y_0\in\T^d$ and consider $\mu_0=\d_{y_0}\in H^{-s-1}(\T^d).$ Again, by the previous representation one has 
$$
v(t,\cdot) = \int_{\T^d}\frac{\d U}{\d m}(t,\cdot,m_0)(y)\dd \mu_0(y)=\frac{\d U}{\d m}(t,\cdot,m_0)(y_0).
$$
Thus, the previous reasonings imply that $H^s(\T^d)\ni m_0\mapsto \frac{\d U}{\d m}(t,\cdot,m_0)(y_0)$ is continuous, uniformly with respect to $(t,y_0)\in [0,T]\times\T^d$, and so the claim follows.

\medskip

(v) Now, let us define $v(\cdot,\cdot,y)$ to be the first component of the solution of \eqref{eq:linearized} with initial condition $\d_y$. 
By linearity of the system \eqref{eq:linearized}, we find 
$$
v(t,x,y) - v(t,x,z) = \frac{\d U}{\d m}(t,x,m_0)(y) -  \frac{\d U}{\d m}(t,x,m_0)(z).
$$
Thus, \eqref{estim:thm-H-s} yields
\begin{align*}
\|v(t,\cdot,y) - v(t,\cdot,z)\|_{H^{-s}} = \Big\|\frac{\d U}{\d m}(t,\cdot,m_0)(y) -  \frac{\d U}{\d m}(t,\cdot,m_0)(z)\Big\|_{H^{-s}}\le C\|\delta_y-\delta_z\|_{H^{-s-1}}\le C|y-z|.
\end{align*}
This shows the Lipschitz continuity of $\T^d\ni y\mapsto \frac{\d U}{\d m}(t,\cdot,m_0)(y)\in H^{-s}(\T^d).$

\medskip

Now, for $\e\in\R$ small and $e_i\in\R^d$ the $i^{th}$ canonical basis element we set $y^i_\e:=y+\e e_i$. 
Again, by linearity of the system \eqref{eq:linearized}, we find 
$$
v(t,x,y^i_\e) - v(t,x,y) = \frac{\d U}{\d m}(t,x,m_0)(y^i_\e) -  \frac{\d U}{\d m}(t,x,m_0)(y).
$$
By dividing this identity by $\e$ and taking the limit as $\e\to 0$, one obtains in the sense of distributions
$$
\partial_{y_i}v(t,x,y)= \partial_{y_i}\frac{\d U}{\d m}(t,x,m_0)(y).
$$
This last object also corresponds to 
$$\left\langle\frac{\d U}{\d m}(t,x,m_0)(y), -\partial_{y_i}\d_y\right\rangle_{H^s,H^{-s}},$$ 
where $\partial_{y_i}\d_y$ stands for the distributional derivative of $\d_y$. This is the same as solving \eqref{eq:linearized} with $\mu_0=-\partial_{y_i}\d_y$.

Therefore, fixing $y,y'\in\T^d$, the linearity of the system \eqref{eq:linearized} \eqref{estim:thm-H-s} yields
\begin{align*}
&\left\|\nabla_y\frac{\d U}{\d m}(t,\cdot,m_0)(y) -  \nabla_{y}\frac{\d U}{\d m}(t,\cdot,m_0)(y')\right\|_{H^{-s}}\le C\left\|D\d_{y} - D\d_{y'} \right\|_{H^{-s-1}}\le C |y - y'|.
\end{align*}
In the last inequality we have relied on the following observation. Since $\widehat{D\d_y}(k)=ik\widehat{\d_y}(k)=\frac{ik}{(2\pi)^d}e^{-ik\cdot y}$, we have
$$
\widehat{D\d_y}(k)- \widehat{D\d_{y'}}(k)=\frac{ik}{(2\pi)^d}(y-y')\cdot(-ik)\int_0^1 e^{-ik\cdot(sy+(1-s)y')}\dd s,
$$
and so
\begin{align*}
\left\|D\d_{y} - D\d_{y'} \right\|_{H^{-s-1}}^2 \le\frac{1}{(2\pi)^{2d}} |y-y'|^2\sum_{k\in\Z^d}\frac{|k|^4}{(1+|k|^2)^{s+1}}\le C|y-y'|^2,
\end{align*}
for some $C>0$ depending only on $d.$ Here, we also used the fact that $s$ satisfies \eqref{hyp:s}.

So, by fixing $(t,m_0)\in [0,T]\times H^s(\T^d)$, we find that $y\mapsto \nabla_y\frac{\d U}{\d m}(t,\cdot,m_0)(y)\in H^{-s}(\T^d)$ is a Lipschitz continuous function. Using similar reasoning, the same conclusion can be made for $y\mapsto D^2_{yy}\frac{\d U}{\d m}(t,\cdot,m_0)(y).$ 
\end{proof}

\section{Well-posedness of the master equation}

In this section we conclude by showing that the master equation \eqref{eq:master_main} has a unique solution. This proof closely follows the proof of  \cite[Theorem 2.4.2]{CarDelLasLio}.

\begin{proof}[Proof of Theorem \ref{thm:main}]
\smallskip

{\bf Existence.}

Let $t_0\in (0,T_{**})$ and $h\in(-1,1)$, $h\neq 0$ such that $t_0+h\in (0,T_{***})$. { Let $(u,m)$ be the solution to the MFG system \eqref{eq:mfg} initiated at some $m_0\in\sQ_R$ at $t_0$ and let $U$ be defined in \eqref{def:U}.} We have
\begin{align*}
\frac{U(t_0+h,x,m_0)-U(t_0,x,m_0)}{h}&=\frac{U(t_0+h,x,m_0)-U(t_0+h,x,m(t_0+h))}{h}\\
&+\frac{U(t_0+h,x,m(t_0+h))-U(t_0,x,m_0)}{h}
\end{align*}
Let us denote $m_s:=(1-s)m_0+s m(t_0+h)$. Now, we can write 
\begin{align*}
&U(t_0+h,x,m(t_0+h))-U(t_0+h,x,m_0)=\int_0^1\int_{\T^d}\frac{\d U}{\d m}(t_0+h,x,m_s)(y)(m(t_0+h,y)-m(t_0,y))\dd y\dd s\\
&=\int_0^1\int_{\T^d}\int_{t_0}^{t_0+h}\frac{\d U}{\d m}(t_0+h,x,m_s)(y)\partial_t m(t,y)\dd t\dd y\dd s\\
&=\int_0^1\int_{\T^d}\int_{t_0}^{t_0+h}\frac{\d U}{\d m}(t_0+h,x,m_s)(y)\left[\Delta m(t,y)+\diver(m(t,y)D_p\cH(t,y,Du(t,y),m(t,y)))\right]\dd t\dd y\dd s\\
\end{align*}

Now, by dividing this expression by $h$ and taking the limit as $h\to 0$, we find that  
\begin{align*}
&\lim_{h\to 0}\frac{U(t_0+h,x,m(t_0+h))-U(t_0+h,x,m_0)}{h}\\
&=\int_0^1\int_{\T^d}\frac{\d U}{\d m}(t_0,x,m_{t_0})(y)\left[\Delta m(t_0,y)+\diver(m(t_0,y)D_p\cH(t_0,y,Du(t_0,y),m(t_0,y)))\right]\dd y\dd s\\
&=\int_{\T^d}\frac{\d U}{\d m}(t_0,x,m_{t_0})(y)\left[\Delta m(t_0,y)+\diver(m(t_0,y)D_p\cH(t_0,y,Du(t_0,y),m(t_0,y)))\right]\dd y\\
&=\int_{\T^d}\nabla_y\cdot[\nabla_w U](t_0,x,m_0)(y)m_0(y)\dd y\\
&-\int_{\T^d}\nabla_w U(t_0,x,m_0)(y)\cdot D_p\cH(t_0,y,Du(t_0,y),m_0(y))m_0(y)\dd y.
\end{align*}
In the first equality, we have used the continuity of the integrand (see Theorem \ref{thm:U_reg}(iv)), while in the last equality,  we used the fact that $\frac{\delta U}{\delta m}(t_0,x,m_{t_0})(\cdot)\in H^s(\T^d)$ (see Theorem \ref{thm:U_reg}(ii)).

Second, by the time regularity of $u$ (see Corollary \ref{cor:diff_u_m}), we have 
\begin{align*}
U(t_0+h,x,m(t_0+h))-U(t_0,x,m_0)=u(t_0+h,x)-u(t_0,x) = h\partial_t u(t_0,x) + o(h).
\end{align*}
Therefore, we have
\begin{align*}
\lim_{h\to 0} \frac{U(t_0+h,x,m_0)-U(t_0,x,m_0)}{h}&=\partial_t u(t_0,x)\\
&-\int_{\T^d}\nabla_y\cdot[\nabla_w U](t_0,x,m_0)(y)m_0(y)\dd y\\
&+\int_{\T^d}\nabla_w U(t_0,x,m_0)(y)\cdot D_p\cH(t_0,y,Du(t_0,y),m_0(y))m_0(y)\dd y
\end{align*}
And since $u$ solves the first equation from \eqref{eq:mfg}, we get
\begin{align}\label{eq:master2}
-\partial_t U(t_0,x,m_0) &- \Delta U(t_0,x,m_0) + \cH(t,x,\nabla U,m)-\int_{\T^d}\nabla_y\cdot[\nabla_w U](t_0,x,m_0)(y)m_0(y)\dd y\\
\nonumber&+\int_{\T^d}\nabla_w U(t_0,x,m_0)(y)\cdot D_p\cH(t_0,y,\nabla U(t_0,y),m_0(y))m_0(y)\dd y=0.
\end{align}

\medskip

Let us remark that for $(t_0,m_0)\in[0,T]\times H^s(\T^d)$ fixed, by Theorem \ref{thm:U_reg}, the mapping
\be\label{eq:nonlocal}
x\mapsto -\int_{\T^d}\nabla_y\cdot[\nabla_w U](t_0,x,m_0)(y)m_0(y)\dd y
+\int_{\T^d}\nabla_w U(t_0,x,m_0)(y)\cdot D_p\cH(t_0,y,\nabla U(t_0,y),m_0(y))m_0(y)\dd y
\ee
has to be regarded a priori as an element of $H^{-s}(\T^d)$. However, by \eqref{eq:master2}, since 
$$x\mapsto -\partial_t U(t_0,x,m_0) - \Delta U(t_0,x,m_0) + \cH(t,x,\nabla U,m)$$ 
is continuous, the mapping in \eqref{eq:nonlocal} becomes also continuous. { Therefore, $U$ is indeed a solution to the master equation.}

\medskip

{\bf Uniqueness.} Let us suppose now that $V:[0,T]\times\T^d\times\sQ_R\to\R$ is another solution to the master equation. Let $m_0\in\sQ_R$. By the regularity of $V$, decreasing the time horizon $T>0$ if necessary, one has that the problem
\begin{align*}
\left\{
\begin{array}{ll}
\partial_t m - \Delta m - \diver (m D_p\cH (t,x,\nabla_x V(t,x,m), m))=0, & {\rm{in}}\ (0,T)\times\T^d,\\
m(t_0,\cdot)= m_0, & {\rm{in}}\ \T^d
\end{array}
\right.
\end{align*} 
has a classical solution. 

%

Now, for all $(t,x)\in[t_0,T]\times\T^d$, set $\tilde u(t,x):= V(t,x,m_t).$ By the regularity of $V$ and $m$ the computations below are justified:

\begin{align*}
\partial_t u(t,x)& = \partial_t V (t,x,m_t) + \int_{\T^d}\frac{\delta V}{\delta m}(t,x,m_t)(y)\partial_t m_t(y)\dd y\\
& = \partial_t V (t,x,m_t) +  \int_{\T^d}{\rm{div}}_y\nabla_w V(t,x,m_t)(y) m_t(y)\dd y\\
& -  \int_{\T^d} \nabla_w V(t,x,m_t)(y)\cdot D_p\cH(t,x,\nabla_x V,m_t) m_t(y)\dd y\\
& = \Delta_x V(t,x,m_t) - \cH(t,x,\nabla_x V(t,x,m_t),m_t)\\
& = \Delta_x \tilde u(t,x) - \cH(t,x,\nabla_x \tilde u(t,x),m_t),
\end{align*}
where in the penultimate equality we have used the fact that $V$ solves the master equation, while in the last equality, we have used the definition of $\tilde u$. Since by definition we also have $\tilde u(T,x)= G(x,m_T)$, we have that $(\tilde u,m)$ is a solution to the mean field game system. That system has a unique solution and therefore, we must have that $V(t_0,x,m_0) = U(t_0,x,m_0)$ for $m_0$ smooth. The uniqueness follows.
\end{proof}

\appendix

\section{A stability result}\label{sec:appendix_estimate}

For a probability density $m$, if we write $m=\bar{m}+\mu,$ with $\mu$ having zero mean (where $\bar m$ stands for the uniform measure on $\T^d$), then we introduce $\tilde{G}$
through  $\tilde{G}(\mu(x),x)=PG(m(x),x),$ where $P$ is the projection operator that removes the mean of a periodic function. In particular, we also have $Pm=\mu.$

Let us recall the Lipschitz continuity condition on $G$, assumed in \eqref{hyp:G-Lip}. For $m,\tilde m$ probability measures, note that $m-\tilde{m}=\mu-\tilde{\mu}.$
We notice that \eqref{hyp:G-Lip} immediately implies
\begin{equation}\nonumber
\|\tilde{G}(\mu_{1},\cdot)-\tilde{G}(\mu_{2},\cdot)\|_{1}^{2}
\leq\Upsilon\|\mu_{1}-\mu_{2}\|_{0}^{2}.
\end{equation}
This follows because $\|Pf\|_{H^{1}}^{2}\leq\|f\|_{H^{1}}^{2}$ for all $f.$

We are now able to state our stability theorem.
\begin{theorem}\label{uniquenessTheorem}
 Let $(u^{1},\bar{m}+\mu^{1})$ and $(u^2,\bar{m}+\mu^{2})$ be two classical solutions of 
\eqref{eq:mfg} on $(0,T)\times\T^d$ , with data $m^{1}(0,\cdot)=m^{1}_{0},$
$m^{2}(0,\cdot)=m^{2}_{0}.$
Assume that there exists $K$ such that the solutions are each bounded by $K:$
\begin{equation}\nonumber
\|Du^{i}\|_{H^{s-1}}+\|\mu^{i}\|_{H^{s-1}} \leq K,\qquad i\in\{1,2\},
\end{equation}
for some $s>2+\frac{d}{2}$ (which is ensured by Theorem \ref{thm:david_old} and the assumption \eqref{hyp:s}).  Assume that \eqref{hyp:H3_original} and \eqref{hyp:G-Lip} hold.  If $T$ is sufficiently small, then there exists a constant $c>0$ (depending on the data and $K$) such that
\begin{equation}\label{eq:stability}
\sup_{t\in[0,T]} \|u^{1}-u^{2}\|_{H^{1}}^{2}+\|m^{1}-m^{2}\|_{L^{2}}^{2}\leq c\|m^{1}_{0}-m^{2}_{0}\|_{L^{2}}^{2}.
\end{equation}
\end{theorem}

\begin{proof}
We use the notation $w^i:=P u^i$ and we notice that $D w^i=D u^i$. We also use the notation 
$$\Theta(t,x,Dw,\mu):=-\cH(t,x,Du,m).$$ 
Let us notice that \eqref{eq:mfg} yields

\begin{equation}\label{eq:mu_w}
\left\{
\begin{array}{ll}
\partial_t w^i + \Delta w^i +P\Theta(t,x,D w^i,\mu^i)=0, & (t,x)\in(0, T)\times\T^d,\\[7pt]
\partial_t \mu^i - \Delta \mu^i + \diver(\mu^i D_p\Theta(t,x,D w^i,\mu^i))+\bar m\diver(D_p\Theta(t,x,D w^i,\mu^i))=0, & (t,x)\in(0,T)\times\T^d,\\[7pt]
\mu^i(t_0,x)=m_0(x)-\bar m,\ \ w^i(\tilde T,x)=\tilde G(x,\mu^i_{T}(x)), & x\in\T^d
\end{array}
\right.
\end{equation}

We define an energy $E=E_{\mu}+E_{w},$ by
\begin{equation}\nonumber
E_{\mu}=\frac{1}{2}\int_{\mathbb{T}^{d}}(\mu^{1}-\mu^{2})^{2}\ dx,
\end{equation}
\begin{equation}\nonumber
E_{w}=\frac{1}{2}\sum_{j=1}^{d}\int_{\mathbb{T}^{d}}(\partial_{x_{j}}w^{1}-\partial_{x_{j}}w^{2})^{2}\ dx.
\end{equation}
We will first bound $E_{\mu}$ and $E_{w}.$  To finish the argument, we will then only need to bound the mean of 
$u_{1}-u_{2}.$

To estimate the growth of this energy, we begin by taking the following time derivative:
\begin{equation}\nonumber
\frac{dE_{\mu}}{dt}=\int_{\mathbb{T}^{d}}(\mu^{1}-\mu^{2})(\partial_t\mu^{1}-\partial_t\mu^{2})\ dx.
\end{equation}
Substituting for $\mu^{1}_{t}$ and $\mu^{2}_{t}$ from \eqref{eq:mu_w}, and then adding 
and subtracting, we arrive at
\begin{align*}
\frac{dE_{\mu}}{dt}&=\int_{\mathbb{T}^{d}}(\mu^{1}-\mu^{2})\Delta(\mu^{1}-\mu^{2}) dx
\\
&- \int_{\mathbb{T}^{d}}(\mu^{1}-\mu^{2})\mathrm{div}\left(
(\mu^{1}-\mu^{2})D_p\Theta(t,x,Dw^1,\mu^1)\right) dx
\\
&- \int_{\mathbb{T}^{d}}(\mu^{1}-\mu^{2})\mathrm{div}\left(
\mu^{2}\left(D_p\Theta(t,x,Dw^1,\mu^1)-D_p\Theta(t,x,Dw^2,\mu^2)\right)\right) dx
\\
&- \bar{m}\int_{\mathbb{T}^{d}}(\mu^{1}-\mu^{2})\mathrm{div}\left(
D_p\Theta(t,x,Dw^1,\mu^1)-D_p\Theta(t,x,Dw^2,\mu^2)\right)dx.
\end{align*}
We expand the derivatives on the right-hand side, introducing the notation
\begin{equation}\nonumber
\frac{dE_{\mu}}{dt}=\sum_{\ell=1}^{14}V_{\ell},
\end{equation}
where the summands are given by the following expressions:
\begin{equation}\nonumber
V_{1}=\int_{\mathbb{T}^{d}}(\mu^{1}-\mu^{2})\Delta(\mu^{1}-\mu^{2})\ dx,
\end{equation}
\begin{equation}\nonumber
V_{2}=- \int_{\mathbb{T}^{d}}(\mu^{1}-\mu^{2})\left(\nabla\left(\mu^{1}-\mu^{2}\right)\right)\cdot
D_p\Theta(t,x,Dw^1,\mu^1)\ dx,
\end{equation}
\begin{equation}\nonumber
V_{3}=- \int_{\mathbb{T}^{d}}\left(\mu^{1}-\mu^{2}\right)^{2}
\mathrm{div}\left(D_p\Theta(t,x,Dw^1,\mu^1)\right)\ dx,
\end{equation}
\begin{equation}\nonumber
V_{4}=- \int_{\mathbb{T}^{d}}(\mu^{1}-\mu^{2})(\nabla\mu^{2})\cdot
\left(D_p\Theta(t,x,Dw^1,\mu^1)-D_p\Theta(t,x,Dw^2,\mu^2)\right)\ dx,
\end{equation}
\begin{equation}\nonumber
V_{5}=- \int_{\mathbb{T}^{d}}(\mu^{1}-\mu^{2})(\mu^{2})\sum_{i=1}^{d}\left[
\Theta_{p_{i}x_{i}}(t,x,Dw^1,\mu^1)-\Theta_{p_{i}x_{i}}(t,x,Dw^2,\mu^2)\right]\ dx,
\end{equation}
\begin{equation}\nonumber
V_{6}=- \int_{\mathbb{T}^{d}}(\mu^{1}-\mu^{2})(\mu^{2})\sum_{i=1}^{d}
\left[\Theta_{p_{i}q}(t,x,Dw^1,\mu^1)\frac{\partial\mu^{1}}{\partial x_{i}}
-\Theta_{p_{i}q}(t,x,Dw^2,\mu^2)\frac{\partial\mu^{1}}{\partial x_{i}}\right]\ dx,
\end{equation}
\begin{equation}\nonumber
V_{7}=- \int_{\mathbb{T}^{d}}(\mu^{1}-\mu^{2})(\mu^{2})\sum_{i=1}^{d}\left[
\Theta_{p_{i}q}(t,x,Dw^2,\mu^2)\left(\frac{\partial(\mu^{1}-\mu^{2})}{\partial x_{i}}\right)
\right]\ dx,
\end{equation}
\begin{align*}\nonumber
V_{8}=- \int_{\mathbb{T}^{d}}(\mu^{1}-\mu^{2})(\mu^{2})\sum_{i=1}^{d}\sum_{j=1}^{d}\Bigg[
\Theta_{p_{i}p_{j}}(t,x,Dw^1,\mu^1)\frac{\partial^{2}w^{1}}{\partial x_{i}\partial x_{j}}
-\Theta_{p_{i}p_{j}}(t,x,Dw^2,\mu^2)\frac{\partial^{2}w^{1}}{\partial x_{i}\partial x_{j}}
\Bigg]\ dx,
\end{align*}
\begin{equation}\nonumber
V_{9}=- \int_{\mathbb{T}^{d}}(\mu^{1}-\mu^{2})(\mu^{2})\sum_{i=1}^{d}\sum_{j=1}^{d}\left[
\Theta_{p_{i}p_{j}}(t,x,Dw^2,\mu^2)\left(\frac{\partial^{2}(w^{1}-w^{2})}{\partial x_{i}\partial x_{j}}\right)
\right]\ dx,
\end{equation}
\begin{equation}\nonumber
V_{10}=- \bar{m}\int_{\mathbb{T}^{d}}(\mu^{1}-\mu^{2})\sum_{i=1}^{d}\left[
\Theta_{p_{i}x_{i}}(t,x,Dw^1,\mu^1)-\Theta_{p_{i}x_{i}}(t,x,Dw^2,\mu^2)\right]\ dx,
\end{equation}
\begin{equation}\nonumber
V_{11}=- \bar{m}\int_{\mathbb{T}^{d}}(\mu^{1}-\mu^{2})\sum_{i=1}^{d}\left[
\Theta_{p_{i}q}(t,x,Dw^1,\mu^1)\frac{\partial\mu^{1}}{\partial x_{i}}
-\Theta_{p_{i}q}(t,x,\mu^{2},Dw^{1})\frac{\partial\mu^{1}}{\partial x_{i}}
\right]\ dx,
\end{equation}
\begin{equation}\nonumber
V_{12}=- \bar{m}\int_{\mathbb{T}^{d}}(\mu^{1}-\mu^{2})\sum_{i=1}^{d}
\Theta_{p_{i}q}(t,x,Dw^2,\mu^2)
\left(\frac{\partial(\mu^{1}-\mu^{2})}{\partial x_{i}}\right)
\ dx,
\end{equation}
\begin{align*}\nonumber
V_{13}=- \bar{m}\int_{\mathbb{T}^{d}}(\mu^{1}-\mu^{2})\sum_{i=1}^{d}\sum_{j=1}^{d}\Bigg[
\Theta_{p_{i}p_{j}}(t,x,Dw^1,\mu^1)\frac{\partial^{2}w^{1}}{\partial x_{i}\partial x_{j}}
-\Theta_{p_{i}p_{j}}(t,x,Dw^2,\mu^2)\frac{\partial^{2}w^{1}}{\partial x_{i}\partial x_{j}}
\Bigg]\ dx,
\end{align*}
\begin{equation}\nonumber
V_{14}=- \bar{m}\int_{\mathbb{T}^{d}}(\mu^{1}-\mu^{2})\sum_{i=1}^{d}\sum_{j=1}^{d}\left[
\Theta_{p_{i}p_{j}}(t,x,Dw^2,\mu^2)
\left(\frac{\partial^{2}(w^{1}-w^{2})}{\partial x_{i}\partial x_{j}}\right)
\right]\ dx.
\end{equation}

We integrate four of these terms by parts.  For $V_{1}$ we arrive at
\begin{equation}\label{V1ByParts}
V_{1}=-\int_{\mathbb{T}^{d}}\left|\nabla\left(\mu^{1}-\mu^{2}\right)\right|^{2}\ dx.
\end{equation}
Continuing to integrate by parts, we have
\begin{equation}\nonumber
V_{2}=\frac{1}{2}\int_{\mathbb{T}^{d}}(\mu^{1}-\mu^{2})^{2}
\mathrm{div}\left(D_p\Theta(t,x,Dw^1,\mu^1)\right)\ dx,
\end{equation}
\begin{equation}\nonumber
V_{7}=\frac{1}{2}\sum_{i=1}^{d}\int_{\mathbb{T}^{d}}(\mu^{1}-\mu^{2})^{2}\frac{\partial}{\partial x_{i}}
\left((\mu^{2})\Theta_{p_{i}q}(t,x,Dw^1,\mu^1)\right)\ dx,
\end{equation}
\begin{equation}\nonumber
V_{12}=\frac{ \bar{m}}{2}\sum_{i=1}^{d}\int_{\mathbb{T}^{d}}(\mu^{1}-\mu^{2})^{2}
\frac{\partial}{\partial x_{i}}\left(\Theta_{p_{i}q}(t,x,Dw^1,\mu^1)\right)\ dx.
\end{equation}
These terms, and also $V_{3},$ are then bounded in terms of the energy, using the bound on the solutions in terms of $K.$
We introduce a  nondecreasing function $\mathcal{G}$ such that
\begin{equation}\label{V-2-3-7-and-12}
V_{2}+V_{3}+V_{7}+V_{12}\leq   \mathcal{G}(K) E_{\mu}.
\end{equation}
Note that the we have estimated, for instance, $\mu^{2}$ in $L^{\infty}$ here, which we may do since $s$ is large enough.
Such estimates will be made several times in the rest of this proof.

For most of the remaining $V_{i}$ terms, we may estimate them using Lipschitz properties of $D_p\Theta$ and its 
derivatives; we then have 
\begin{equation}\label{V-many}
V_{4}+V_{5}+V_{6}+V_{8}+V_{10}+V_{11}+V_{13}
\leq   \mathcal{G}(K)(E_{\mu}+E_{\mu}^{1/2}E_{w}^{1/2}),
\end{equation}
where $\mathcal{G}(K)$ is as before.

For the final two terms, $V_{9}$ and $V_{14},$  we use Young's inequality with an eye toward bounding them
by a beneficial term arising from $E_{w}.$
For $V_{9},$ we bound $\Theta_{p_{i}p_{j}}$ and $\mu^{2}$ with $\mathcal{G}(K):$
\begin{equation}\nonumber
V_{9}\leq  \mathcal{G}(K)\sum_{i=1}^{d}\sum_{j=1}^{d}\int_{\mathbb{T}^{d}}(\mu^{1}-\mu^{2})
\frac{\partial^{2}(w^{1}-w^{2})}{\partial_{x_{i}}\partial_{x_{j}}}\ dx.
\end{equation}
We introduce a constant $\varpi=\frac{1}{8\Upsilon}.$
We then apply Young's inequality, with parameter $4  \mathcal{G}(K)/\varpi:$
\begin{multline}\label{V-9}
V_{9} \leq \mathcal{G}(K)\int_{\mathbb{T}^{d}}(\mu^{1}-\mu^{2})^{2}\ dx
+\frac{1}{8}\sum_{j=1}^{d}\int_{\mathbb{T}^{d}}(\partial_{x_{j}}(Dw^{1}-Dw^{2}))^{2}\ dx
\\
\leq \mathcal{G}(K)E_{\mu}
+\frac{\varpi}{8}\sum_{j=1}^{d}\int_{\mathbb{T}^{d}}(\partial_{x_{j}}(Dw^{1}-Dw^{2}))^{2}\ dx.
\end{multline}
Notice that in the first term on the right-hand side, we have simply incorporated $\varpi$ into $\mathcal{G}.$
The final term, $V_{14},$ is then entirely similar:
\begin{equation}\label{V-14}
V_{14}\leq 
\mathcal{G}(K)E_{\mu}
+\frac{\varpi}{8}\sum_{j=1}^{d}\int_{\mathbb{T}^{d}}(\partial_{x_{j}}(Dw^{1}-Dw^{2}))^{2}\ dx.
\end{equation}

We add the bounds \eqref{V-2-3-7-and-12}, \eqref{V-many}, \eqref{V-9}, and
\eqref{V-14}, finding the following conclusion:
\begin{equation}\label{sumOfMostVTerms}
\sum_{\ell=2}^{14}V_{\ell}\leq  \mathcal{G}(K)(E_{w}+E_{\mu}) + 
\frac{\varpi}{4}\sum_{j=1}^{d}\int_{\mathbb{T}^{d}}(\partial_{x_{j}}(Dw^{1}-Dw^{2}))^{2}\ dx.
\end{equation}

We next treat $E_{w},$ introducing the notation
$E_{w}=\sum_{j=1}^{d}E^{j}_{w},$ 
with
\begin{equation}\nonumber
\frac{dE^{j}_{w}}{dt}=\int_{\mathbb{T}^{d}}(\partial_{x_{j}}w^{1}-\partial_{x_{j}}w^{2})
\partial_{t}(\partial_{x_{j}}w^{1}-\partial_{x_{j}}w^{2})\ dx.
\end{equation}
We further decompose each of these time derivatives as
\begin{equation}\nonumber
\frac{dE^{j}_{w}}{dt}=\sum_{\ell=1}^{6}W^{j}_{\ell},
\end{equation}
where the $W^{j}_{i}$ are given by
\begin{equation}\nonumber
W^{j}_{1}=-\int_{\mathbb{T}^{d}}(\partial_{x_{j}}w^{1}-\partial_{x_{j}}w^{2})
\Delta(\partial_{x_{j}}w^{1}-\partial_{x_{j}}w^{2})\ dx,
\end{equation}
\begin{equation}\nonumber
W^{j}_{2}=- \int_{\mathbb{T}^{d}}(\partial_{x_{j}}w^{1}-\partial_{x_{j}}w_{2})
\left(\Theta_{x_{j}}(t,x,Dw^1,\mu^1)-\Theta_{x_{j}}(t,x,Dw^2,\mu^2)\right)\ dx,
\end{equation}
\begin{equation}\nonumber
W^{j}_{3}=- \int_{\mathbb{T}^{d}}(\partial_{x_{j}}w^{1}-\partial_{x_{j}}w^{2})
\left(\Theta_{q}(t,x,Dw^1,\mu^1)\mu^{1}_{x_{j}}-\Theta_{q}(t,x,Dw^2,\mu^2)\mu^{1}_{x_{j}}\right)\ dx,
\end{equation}
\begin{equation}\nonumber
W^{j}_{4}=- \int_{\mathbb{T}^{d}}(\partial_{x_{j}}w^{1}-\partial_{x_{j}}w^{2})
\left(D_p\Theta(t,x,Dw^2,\mu^2)\mu^{1}_{x_{j}}-\Theta_{q}(t,x,Dw^2,\mu^2)\mu^{2}_{x_{j}}\right)\ dx,
\end{equation}
\begin{equation}\nonumber
W^{j}_{5}=- \sum_{i=1}^{d}\int_{\mathbb{T}^{d}}(\partial_{x_{j}}w^{1}-\partial_{x_{j}}w^{2})
\left(\Theta_{p_{i}}(t,x,Dw^1,\mu^1)\partial^{2}_{x_{i}x_{j}}w^{1}
-\Theta_{p_{i}}(t,x,Dw^2,\mu^2)\partial^{2}_{x_{i}x_{j}}w^{1}\right)\ dx,
\end{equation}
\begin{equation}\nonumber
W^{j}_{6}=- \sum_{i=1}^{d}\int_{\mathbb{T}^{d}}(\partial_{x_{j}}w^{1}-\partial_{x_{j}}w^{2})
\left(\Theta_{p_{i}}(t,x,Dw^2,\mu^2)\partial^{2}_{x_{i}x_{j}}w^{1}
-\Theta_{p_{i}}(t,x,Dw^2,\mu^2)\partial^{2}_{x_{i}x_{j}}w^{2}\right)\ dx.
\end{equation}

As before, we begin by integrating some of the terms by parts.  Integrating $W^{j}_{1}$ by parts, we have
\begin{equation}\nonumber
W^{j}_{1}=\int_{\mathbb{T}^{d}}\left|\nabla\partial_{x_{j}}\left(w^{1}-w^{2}\right)\right|^{2}\ dx.
\end{equation}
We also have the following, upon integrating $W_{6}^{j}$ by parts:
\begin{equation}\nonumber
W_{6}^{j}=\frac{1}{2}\sum_{i=1}^{d}\int_{\mathbb{T}^{d}}
(\partial_{x_{j}}w^{1}-\partial_{x_{j}}w^{2})^{2}\partial_{x_{i}}\left(\Theta_{p_{i}}(t,x,Dw^2,\mu^2)\right)\ dx.
\end{equation}

With $\mathcal{G}$ as before, we may then estimate several of the terms:
\begin{equation}\label{W-2-3-5-and-6}
W_{2}^{j}+W_{3}^{j}+W_{5}^{j}+W_{6}^{j}
\leq
  \mathcal{G}(K)(E_{w}+E_{w}^{1/2}E_{\mu}^{1/2}),
\end{equation}

The only remaining term to estimate is $W_{4}^{j}.$    We first bound $D_p\Theta(t,x,Dw^2,\mu^2)$
in $L^{\infty}$ using $\mathcal{G}(K),$ arriving at
\begin{equation}\nonumber
W_{4}^{j}\leq   \mathcal{G}(K)\int_{\mathbb{T}^{d}}(\partial_{x_{j}}w^{1}-\partial_{x_{j}}w^{2})
\partial_{x_{j}}(\mu^{1}-\mu^{2})\ dx.
\end{equation}
Applying Young's inequality with $2  \mathcal{G}(K)$ as the parameter, we have
\begin{equation}\nonumber
W_{4}^{j}\leq
 \mathcal{G}(K)E_{w} 
+ \frac{1}{4}\int_{\mathbb{T}^{d}}(\partial_{x_{j}}\mu^{1}-\partial_{x_{j}}\mu^{2})^{2}\ dx.
\end{equation}

Adding  \eqref{W-2-3-5-and-6} to this yields
\begin{equation}\label{sumOfMostWTerms+V1}
\sum_{\ell=2}^{6}\sum_{j=1}^{d}W_{\ell}^{j}
\leq  \mathcal{G}(K)(E_{w}+E_{\mu})
+\frac{1}{4}\int_{\mathbb{T}^{d}}|D\mu^{1}-D\mu^{2}|^{2}\ dx.
\end{equation}

We are now in a position to integrate $\frac{dE_{\mu}}{dt}$ with respect to time over the interval $[0,t];$ doing so, 
we arrive at
\begin{equation}\nonumber
E_{\mu}(t)=E_{\mu}(0)+\int_{0}^{t}\sum_{\ell=1}^{14}V_{\ell}.
\end{equation}
Applying \eqref{sumOfMostVTerms} then yields
\begin{align*}
E_{\mu}(t)\leq E_{\mu}(0)+  T \mathcal{G}(K)(E_{\mu}+E_{w})
+\int_{0}^{t}\left[V_{1}
+\frac{\varpi}{4}\sum_{j=1}^{d}\int_{\mathbb{T}^{d}}(\partial_{x_{j}}(Dw^{1}-Dw^{2}))^{2}\ dx\right]\ d\tau.
\end{align*}
We similarly integrate $\frac{dE_{w}}{dt}$ over the temporal interval $[t,T],$ finding
\begin{equation}\nonumber
E_{w}(t)=E_{w}(T)-\int_{t}^{T}\sum_{j=1}^{d}\sum_{\ell=1}^{6}W^{j}_{\ell}\ d\tau.
\end{equation}
Applying \eqref{sumOfMostWTerms+V1}, we find
\begin{equation}\label{almostVarpi}
E_{w}(t)\leq E_{w}(T)
+  T\mathcal{G}(K)(E_{w}+E_{\mu})
-\sum_{j=1}^{d}\int_{t}^{T}W_{1}^{j}\ d\tau
+\frac{1}{4}\int_{t}^{T}\int_{\mathbb{T}^{d}}|D\mu^{1}-D\mu^{2}|^{2}\ dxd\tau.
\end{equation}

We work now with the term $E_{w}(T).$  A formula for this is
\begin{equation}\nonumber
E_{w}(T)=\sum_{j}\|\partial_{x_{j}}w_{1}(T,\cdot)-\partial_{x_{j}}w_{2}(T,\cdot)\|_{0}^{2}.
\end{equation}
Using the payoff function, $\tilde{G},$ this can be estimated as
\begin{equation}\nonumber
E_{w}(T)\leq\|\tilde{G}(\mu_{1}(T,\cdot),\cdot)-\tilde{G}(\mu_{2}(T,\cdot),\cdot)\|_{1}^{2}.
\end{equation}
Using our smoothing assumption on $\tilde{G},$ then, this may again be estimated as
\begin{equation}\nonumber
E_{w}(T)\leq\Upsilon\|\mu_{1}(T,\cdot)-\mu_{2}(T,\cdot)\|_{0}^{2}.
\end{equation}
We may replace this with a supremum, as
\begin{equation}\nonumber
E_{w}(T)\leq \Upsilon\sup_{t\in[0,T]}E_{\mu}(t).
\end{equation}

We  multiply \eqref{almostVarpi} by $\varpi>0,$ apply the definitions of $V_{1}$ and $W^{j}_{1},$
and summarize what we have found thus far:
\begin{multline}\label{almostDoneUniqueness}
\varpi E_{w}(t)+E_{\mu}(t)+\int_{0}^{t}\|D\mu^{1}-D\mu^{2}\|_{0}^{2}\ d\tau
+\varpi\int_{t}^{T}\|D^{2}w^{1}-D^{2}w^{2}\|_{0}^{2}\ d\tau
\\
\leq
\varpi\Upsilon\sup_{t\in[0,T]}E_{\mu}(t)
+E_{\mu}(0)+  T\mathcal{G}(K)(E_{w}(t)+E_{\mu}(t))
\\
+\frac{\varpi}{4}\int_{t}^{T}\|D\mu^{1}-D\mu^{2}\|_{0}^{2}\ d\tau
+\frac{\varpi}{4}\int_{0}^{t}\|D^{2}w^{1}-D^{2}w^{2}\|_{0}^{2}\ d\tau.
\end{multline}
We can bound the integrals appearing on the right-hand side by taking a larger domain of integration:
\begin{multline}\label{almostDoneUniqueness2}
\varpi E_{w}(t)+E_{\mu}(t)+\int_{0}^{t}\|D\mu^{1}-D\mu^{2}\|_{0}^{2}\ d\tau
+\varpi\int_{t}^{T}\|D^{2}w^{1}-D^{2}w^{2}\|_{0}^{2}\ d\tau
\\
\leq
\varpi\Upsilon\sup_{t\in[0,T]}E_{\mu}(t)
+E_{\mu}(0)+  T\mathcal{G}(K)(E_{w}(t)+E_{\mu}(t))
\\
+\frac{\varpi}{4}\int_{0}^{T}\|D\mu^{1}-D\mu^{2}\|_{0}^{2}\ d\tau
+\frac{\varpi}{4}\int_{0}^{T}\|D^{2}w^{1}-D^{2}w^{2}\|_{0}^{2}\ d\tau.
\end{multline}
Considering the first integral on the left-hand side of \eqref{almostDoneUniqueness2}, we find
\begin{multline}\label{workaroundParabolic1}
\int_{0}^{t}\|D\mu^{1}-D\mu^{2}\|_{0}^{2}\ d\tau
\leq
\varpi\Upsilon\sup_{t\in[0,T]}E_{\mu}(t)
+E_{\mu}(0)+  T\mathcal{G}(K)(E_{w}(t)+E_{\mu}(t))
\\
+\frac{\varpi}{4}\int_{0}^{T}\|D\mu^{1}-D\mu^{2}\|_{0}^{2}\ d\tau
+\frac{\varpi}{4}\int_{0}^{T}\|D^{2}w^{1}-D^{2}w^{2}\|_{0}^{2}\ d\tau.
\end{multline}
Taking the supremum with respect to $t$ on both sides of \eqref{workaroundParabolic1}, we have
\begin{multline}\label{workaroundParabolic2}
\int_{0}^{T}\|D\mu^{1}-D\mu^{2}\|_{0}^{2}\ d\tau
\leq
\varpi\Upsilon\sup_{t\in[0,T]}E_{\mu}(t)+E_{\mu}(0)
+  T\mathcal{G}(K)\left(\sup_{t\in[0,T]}(E_{w}(t)+E_{\mu}(t))\right)
\\
+\frac{\varpi}{4}\int_{0}^{T}\|D\mu^{1}-D\mu^{2}\|_{0}^{2}\ d\tau
+\frac{\varpi}{4}\int_{0}^{T}\|D^{2}w^{1}-D^{2}w^{2}\|_{0}^{2}\ d\tau.
\end{multline}
Now considering instead the second integral on the left-hand side of \eqref{almostDoneUniqueness2}, making the
corresponding manipulations, we have
\begin{multline}\label{workaroundParabolic3}
\varpi\int_{0}^{T}\|D^{2}w^{1}-D^{2}w^{2}\|_{0}^{2}\ d\tau
\leq
\varpi\Upsilon\sup_{t\in[0,T]}E_{\mu}(t)+E_{\mu}(0)
+  T\mathcal{G}(K)\left(\sup_{t\in[0,T]}(E_{w}(t)+E_{\mu}(t))\right)
\\
+\frac{\varpi}{4}\int_{0}^{T}\|D\mu^{1}-D\mu^{2}\|_{0}^{2}\ d\tau
+\frac{\varpi}{4}\int_{0}^{T}\|D^{2}w^{1}-D^{2}w^{2}\|_{0}^{2}\ d\tau.
\end{multline}
We assume that $\varpi$ satisfies $\varpi<1.$
Adding \eqref{workaroundParabolic2} and \eqref{workaroundParabolic3} and rearranging, 
we have
\begin{multline}\nonumber
\int_{0}^{T}\|D\mu^{1}-D\mu^{2}\|_{0}^{2}\ d\tau
+\varpi\int_{0}^{T}\|D^{2}w^{1}-D^{2}w^{2}\|_{0}^{2}\ d\tau
\\
\leq 4E_{\mu}(0)+4\varpi\Upsilon E_{\mu}(0)
+  T\mathcal{G}(K)\left(\sup_{t\in[0,T]}(E_{w}(t)+E_{\mu}(t))\right).
\end{multline}
Combining this with \eqref{almostDoneUniqueness2}, and making some adjustments of factors of $\varpi,$ we have
\begin{multline}\label{finalUniqueness}
\sup_{t\in[0,T]}(\varpi E_{w}(t)+E_{\mu}(t))
\leq 2E_{\mu}(0)
+2\varpi\Upsilon\sup_{t\in[0,T]}\left(\varpi E_{w}(t)+E_{\mu}(t)\right)
\\
+  T\mathcal{G}(K)\left(\sup_{t\in[0,T]}(\varpi E_{w}(t)+E_{\mu}(t))\right).
\end{multline}

Recall our choice $\varpi=\frac{1}{8\Upsilon}.$
We take 
$  T \mathcal{G}(K)<\frac{1}{4}$ (which may be accomplished 
simply by taking $T$ sufficiently small), concluding
\begin{equation}\nonumber
\sup_{t\in[0,T]}\left(E_{\mu}(t)+\varpi E_{w}(t)\right)\leq 8E_{\mu}(0)=8\|m^{1}(0)-m^{2}(0)\|_{0}^{2}.
\end{equation}
We may eliminate the factor of $\varpi$ on the left, resulting in
\begin{equation}\label{allButTheMean}
\sup_{t\in[0,T]}\left(E_{\mu}(t)+E_{w}(t)\right)\leq 64\Upsilon\|m^{1}(0)-m^{2}(0)\|_{0}^{2}.
\end{equation}
Let us notice that the estimate on the $\mu$ variables directly translates to the estimate on the $m$ variables (since $m^1-m^2=\mu^1=\mu^2$). Therefore, to conclude with the estimate on the $u$ variables from the estimate on the $w$ variables, it remains to estimate the difference of the averages of the $u$ variables.

The equation for $u$ is
\begin{equation}\nonumber
\partial_{t}u=-\Delta u + \cH(t,x,\nabla u, m).
\end{equation}
Integrating this over the entire spatial domain, we have
\begin{equation}\nonumber
\partial_{t}\left(\int_{\mathbb{T}^{d}}u\ dx\right)=\int_{\mathbb{T}^{d}}\cH(t,x,\nabla u,m)\ dx.
\end{equation}
Let $(u_{1},m_{1})$ and $(u_{2},m_{2})$ be two solutions.  We take the difference in the evolution equations for the mean,
recalling that $\nabla u_{i}=\nabla w_{i}:$
\begin{equation}\nonumber
\partial_{t}\int_{\mathbb{T}^{d}}u_{1}-u_{2}\ dx = \int_{\mathbb{T}^{d}}\cH(t,x,\nabla w_{1},m_{1})-\cH(t,x,\nabla w_{2},m_{2})\ dx.
\end{equation}
Taking absolute values, adding and subtracting, and using the triangle inequality, we have
\begin{multline}\nonumber
\left|\partial_{t}\int_{\mathbb{T}_{d}}u_{1}-u_{2}\ dx\right|
\leq
\int_{\mathbb{T}^{d}}\left|\cH(t,x,\nabla w_{1},m_{1})-\cH(t,x,\nabla w_{2},m_{1})\right|\ dx
\\
+\int_{\mathbb{T}^{d}}\left|\cH(t,x,\nabla w_{2},m_{1})-\cH(t,x,\nabla w_{2},m_{2})\right|\ dx.
\end{multline}
Using the Lipschitz properties of the Hamiltonian, and the boundedness of the solutions $(w_{i},m_{i}),$ we may bound
this as
\begin{equation}\nonumber
\left|\partial_{t}\int_{\mathbb{T}^{d}}u_{1}-u_{2}\ dx\right|\leq c\int_{\mathbb{T}^{d}}|\nabla w_{1}-\nabla w_{2}|\ dx
+c\int_{\mathbb{T}^{d}}|m_{1}-m_{2}|\ dx.
\end{equation}
Using $m_{1}-m_{2}=\mu_{1}-\mu_{2},$ this then implies
\begin{equation}\label{mostOfMean}
\left|\partial_{t}\int_{\mathbb{T}^{d}}u_{1}-u_{2}\ dx\right|\leq
c\left(\|w_{1}-w_{2}\|_{H^{1}}+\|\mu_{1}-\mu_{2}\|_{L^{2}}\right).
\end{equation}

We can write the difference of the means as
\begin{equation}\nonumber
\int_{\mathbb{T}^{d}}u_{1}(t,\cdot)-u_{2}(t,\cdot)\ dx 
= \int_{\mathbb{T}^{d}}u_{1}(T,\cdot)-u_{2}(T,\cdot)\ dx 
-\int_{t}^{T}\partial_{t}\left(\int_{\mathbb{T}^{d}}u_{1}-u_{2}\ dx\right)ds.
\end{equation}
Using the terminal condition for $u,$ this becomes
\begin{equation}\nonumber
\int_{\mathbb{T}^{d}}u_{1}(t,\cdot)-u_{2}(t,\cdot)\ dx
=\int_{\mathbb{T}^{d}}G(m_{1}(T,\cdot),\cdot)-G(m_{2}(T,\cdot),\cdot)\ dx
-\int_{t}^{T}\partial_{t}\left(\int_{\mathbb{T}^{d}}u_{1}-u_{2}\ dx\right)ds.
\end{equation}
We may then estimate this as
\begin{equation}\nonumber
\left|\int_{\mathbb{T}^{d}}u_{1}-u_{2}\ dx\right|\leq\int_{\mathbb{T}^{d}}|G(m_{1}(T,\cdot),\cdot)-G(m_{2}(T,\cdot),\cdot)|\ dx
+T\sup_{t\in[0,T]}\left|\partial_{t}\int_{\mathbb{T}^{d}}u_{1}-u_{2}\ dx\right|.
\end{equation}
The Lipschitz property of $G$ then implies
\begin{equation}\nonumber
\left|\int_{\mathbb{T}^{d}}u_{1}-u_{2}\ dx\right|\leq c\int_{\mathbb{T}^{d}}|m_{1}(T,\cdot)-m_{2}(T,\cdot)|\ dx
+T\sup_{t\in[0,T]}\left|\partial_{t}\int_{\mathbb{T}^{d}}u_{1}-u_{2}\ dx\right|.
\end{equation}
Again using $m_{1}-m_{2}=\mu_{1}-\mu_{2},$ and also using \eqref{mostOfMean}, and using a supremum in time,
we have
\begin{equation}\nonumber
\sup_{t\in[0,T]}\left|\int_{\mathbb{T}^{d}}u_{1}-u_{2}\ dx\right|
\leq c\sup_{t\in[0,T]}(\|w_{1}-w_{2}\|_{H^{1}}+\|\mu_{1}-\mu_{2}\|_{L^{2}}).
\end{equation}
We see that \eqref{allButTheMean} then implies 
\begin{equation}\nonumber
\sup_{t\in[0,T]}\left|\int_{\mathbb{T}^{d}}u_{1}-u_{2}\ dx\right|
\leq c\|m^{1}_{0}-m^{2}_{0}\|.
\end{equation}
This completes the proof of \eqref{eq:stability}.
\end{proof}

\end{document}